\documentclass[11pt]{amsart}
\usepackage{amsmath}
\usepackage{amsthm}
\usepackage{amsfonts}
\usepackage{graphicx}

\newtheorem{theorem}{Theorem}[section]
\newtheorem{corollary}[theorem]{Corollary}
\newtheorem{lemma}[theorem]{Lemma}
\newtheorem{proposition}[theorem]{Proposition}

\theoremstyle{definition}
\newtheorem{definition}[theorem]{Definition}
\theoremstyle{remark}
\newtheorem{remark}[theorem]{Remark}
\numberwithin{equation}{section}
\newcommand{\cone}{\mbox{$\times \hspace*{-0.244cm} \times$}}

\begin{document}

\title[The Periodic Plateau problem and its application]
{The Periodic Plateau problem and its application}
\author[ J. CHOE]{JAIGYOUNG CHOE}
\date{(arXiv) April 19, 2021.\,\,\,\,\,(Revised) April 26, 2021}
\thanks{Supported in part by NRF-2018R1A2B6004262}

\address{Korea Institute for Advanced Study, Seoul, 02455, Korea}
\email{choe@kias.re.kr}

\begin{abstract} Given a noncompact disconnected complete periodic curve $\Gamma$ with no self intersection in $\mathbb R^3$, it is proved that there exists a noncompact simply connected periodic minimal surface spanning $\Gamma$. As an application it is shown that for any tetrahedron $T$ with dihedral angles $\leq90^\circ$ there exist four embedded minimal annuli in $T$ which are perpendicular to $\partial T$ along their boundary.  It is also proved that every Platonic solid of $\mathbb R^3$ contains five types of free boundary embedded minimal surfaces of genus zero.\\

\noindent{\it Keywords}: Plateau problem, periodic, minimal surface, free boundary, Platonic solid\\
{\it MSC}\,: 53A10, 49Q05
\end{abstract}

\maketitle

\section{introduction}
The famous problem of finding a surface of least area spanning a given Jordan curve, called the Plateau problem, was settled by Douglas and Rad\'{o} independently in 1931. Since then many questions have been raised about the Douglas-Rad\'{o} solution: the uniqueness, the embeddedness, the topology of the solution and the number of solutions.

In this paper we are concerned with the Plateau problem for a noncompact disconnected complete curve $\Gamma\subset\mathbb R^3$ which is periodic. $\Gamma$ is said to be periodic if $\Gamma$ has a fundamental piece $\bar{\gamma}$ in a convex polyhedron $U$ and $\Gamma$ is the infinite union of the congruent copies of $\bar{\gamma}$ obtained in a periodic way. In particular, $\Gamma$ is {\it helically periodic} if it is the union of images of $\bar{\gamma}$ under the cyclic group $\langle\sigma\rangle$ generated by a screw motion  $\sigma$. $\Gamma$ is  {\it translationally periodic} if it is invariant under the cyclic group $\langle\tau\rangle$ generated by a  translation $\tau$. $\Gamma$ is {\it rotationally periodic} if the congruent copies of $\bar{\gamma}$ are obtained by repeatedly extending $\bar{\gamma}$ through $180^\circ$-rotations about the lines connecting each pair of endpoints of $\bar{\gamma}$. $\Gamma$ is {\it reflectively periodic} if the congruent copies of $\bar{\gamma}$ are obtained by infinitely extending $\bar{\gamma}$ by the reflections across the planar faces of $\partial U$ (see Figure 1). The extensions by screw motions, translations, rotations and reflections are to be performed infinitely until $\Gamma$ becomes complete.

We prove that for every complete noncompact disconnected periodic curve $\Gamma$ in $\mathbb R^3$ there exists a noncompact simply connected minimal surface $\Sigma\subset\mathbb R^3$ spanning $\Gamma$ such that $\Sigma$ inherits the periodicity of $\Gamma$ (Theorem \ref{main}). Furthermore, in case $\Gamma$ consists of the $x_3$-axis and a complete connected translationally periodic curve $\gamma_1$ winding around the $x_3$-axis such that a fundamental piece of $\gamma_1$ admits a one-to-one orthogonal projection onto a convex closed curve in the $x_1x_2$-plane, we can show that $\Sigma$ is unique and embedded (Theorem \ref{plateau}).
\begin{center}
\includegraphics[width=5.85in]{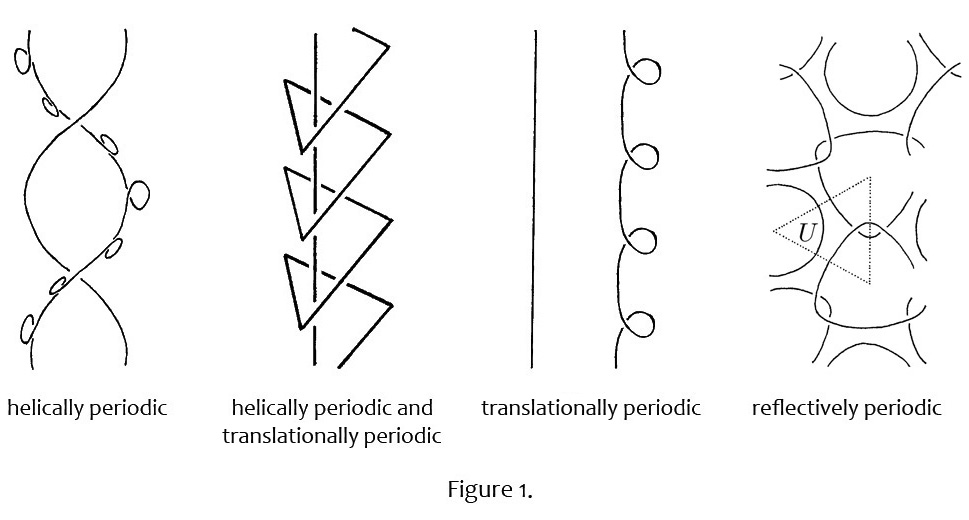}
\end{center}

These two theorems have an interesting application. Smyth \cite{Sm} showed that given a tetrahedron $T$, there exist three embedded minimal disks in $T$ which meet $\partial T$ orthogonally along their boundary. From $T$ Smyth considered a quadrilateral $\Gamma$ whose edges are perpendicular to the faces of $T$. $\Gamma$ bounds a unique minimal graph $\Sigma$. He then showed that the conjugate minimal surface of $\Sigma$ is the desired minimal surface in $T$.

In this paper we will first see that the tetrahedron $T$ gives rise to a noncompact, disconnected, translationally periodic, piecewise linear curve $\Gamma$ such that the edges (=line segments) of a fundamental piece $\bar{\gamma}$ of $\Gamma$ are perpendicular to the faces of $T$. In fact, $\bar{\gamma}$ has two components $\bar{\gamma}_0$, $\bar{\gamma}_1$, where $\bar{\gamma}_0$ has only one edge and $\bar{\gamma}_1$ has 3 edges. So one of the two components of $\Gamma$ is a straight line $\ell$. By Theorem \ref{main} $\Gamma$ bounds a noncompact simply connected translationally periodic minimal surface $\Sigma$. Let $\Sigma^*$ be its conjugate minimal surface. In Theorem \ref{fb} we will prove that if $\ell$ is properly chosen relative to $\bar{\gamma}_1$ then $\Sigma^*$ is a minimal annulus in $T$ which is perpendicular to $\partial T$ (see Figure 2). One boundary component of $\Sigma^*$ is a convex closed curve lying in one face of $T$
and the other component traces along the remaining three faces. Since there are four lines perpendicular to a face of $T$ we conclude that there exist
\begin{center}
\includegraphics[width=4.8in]{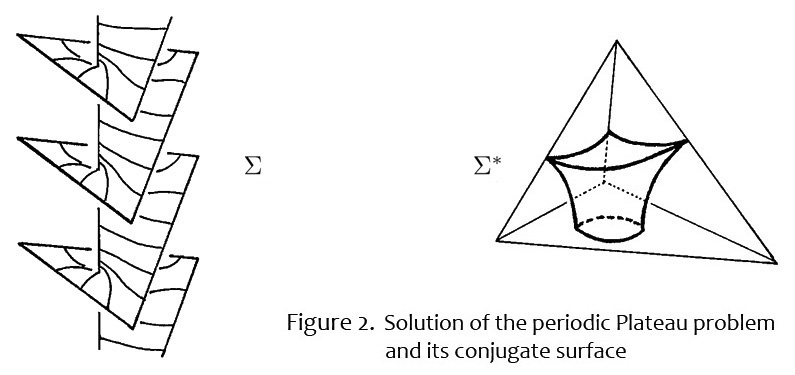}
\end{center}
four free boundary minimal annuli in $T$ if the dihedral angles of $T$ are $\leq90^\circ$. If at least one dihedral angle of $T$ is $>90^\circ$, there exist four minimal annuli which are not necessarily inside $T$ but still perpendicular to the planes containing the faces of $T$ along their boundary.

In general, one cannot generalize Theorem \ref{fb} to construct a free boundary minimal annulus in a polyhedron other than a tetrahedron. However, in case $P_y$ is a right pyramid with a regular polygonal base $B$ and apex $p$ (i.e., $P_y=p\cone B$, the cone), we can show the existence of a free boundary minimal annulus $\Sigma^*$ in $P_y$ such that one boundary component of $\Sigma^*$ is in $B$ and the other component in $p\cone \partial B$ winding around $p$ (Theorem \ref{ps}). Consequently, it is proved that every Platonic solid $P_s$ bounded by regular $n$-gons contains five types of free boundary embedded minimal surfaces $\Sigma_1,\ldots,\Sigma_5$ of genus 0. Three of them, $\Sigma_1,\Sigma_2,\Sigma_3$, intersect each face of $P_s$ along $1,n,2n$ convex closed congruent curves, respectively. $\Sigma_4$ intersects every edge of $P_s$, while $\Sigma_5$ surrounds every vertex of $P_s$ (Corollary \ref{pl}; see Figure 3). As a matter of fact, if $P_s$ is the cube, $\Sigma_1$ is the well-known Schwarz $P$-surface, $\Sigma_4$ is Neovius' surface and $\Sigma_5$ is Schoen's I-WP surface. Finally, if $P_r$ is a right pyramid whose base is a rhombus,  a free boundary minimal annulus in $P_r$ can be similarly constructed (Corollary \ref{pr}).

\begin{center}
\includegraphics[width=5.5in]{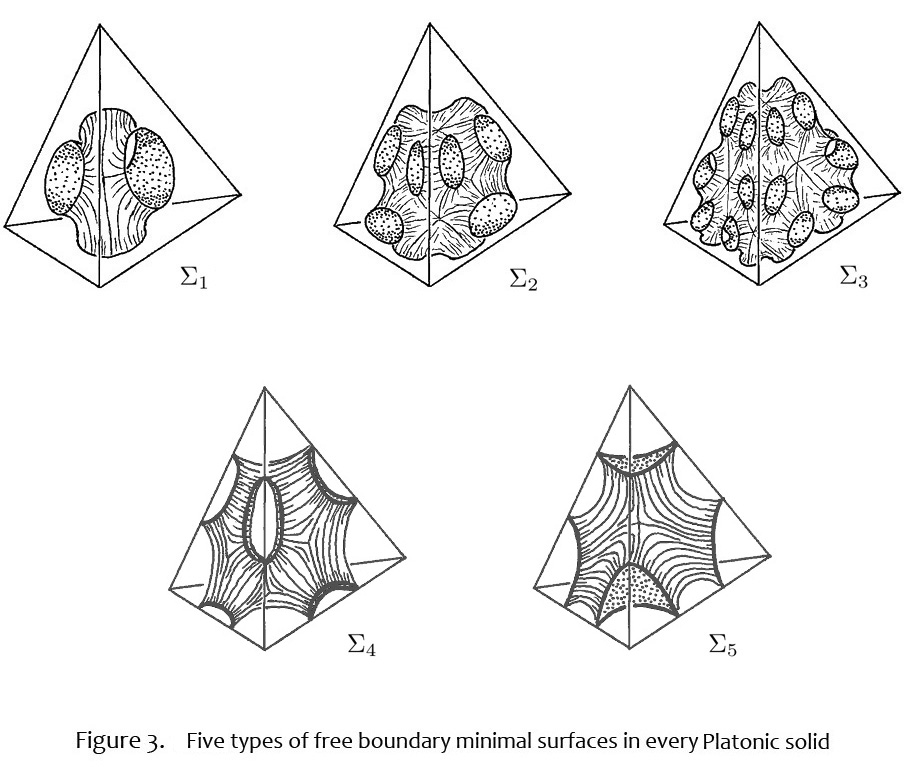}
\end{center}

\section{Periodic Plateau problem}
A Jordan curve is simple and closed. So it has no self intersection and is homeomorphic to a circle. If a simple curve $\Gamma\subset\mathbb R^3$ is not closed but homeomorphic to $\mathbb R^1$ and has infinite length, one cannot in general find a minimal surface spanning $\Gamma$. However, if there exists a surface of finite area spanning $\Gamma$, one can easily show the existence of a minimal surface spanning $\Gamma$. The same is true if $\Gamma$ is the union of simple open curves of infinite lengths bounding a surface of finite area. In case $\Gamma$ cannot bound a surface of finite area, one needs to impose extra conditions on $\Gamma$ to get a minimal surface spanning $\Gamma$. In this section we will see that the periodicity of $\Gamma$ is a sufficient condition for this purpose.
\begin{definition}
Let $\Gamma\subset\mathbb R^3$ be the union of complete open rectifiable curves $\gamma_1,\gamma_2,\gamma_3,\ldots$ and let $U$ be a convex polyhedral domain in $\mathbb R^3$. $\Gamma$ is said to be {\it periodic} if $\Gamma$ is the infinite union of the congruent copies of $\bar{\gamma}:=\Gamma\cap U$. $\bar{\gamma}$ is called a {\it fundamental piece} of $\Gamma$.

a) Suppose $\Gamma$ is homeomorphic to two parallel lines. $\Gamma$ is {\it translationally periodic} if it is the union of translated fundamental pieces $\tau^n(\bar{\gamma})$ for the cyclic group $\langle\tau\rangle$ generated by a parallel translation $\tau$.  $\Gamma$ is invariant under $\langle\tau\rangle$. Moreover, $\Gamma$ is {\it helically periodic} if it is the union of $\sigma^n(\bar{\gamma})$ for the cyclic group $\langle\sigma\rangle$ generated by a screw motion $\sigma$. Assume that the screw motion $\sigma$ is the rotation about the $x_3$-axis by angle $\beta$ composed with the translation  by $e$, that is,
\begin{equation}\label{heli}
\sigma(r\cos\theta,r\sin\theta,x_3)=(r\cos(\theta+\beta),r\sin(\theta+\beta),x_3+e).
\end{equation}
Every translationally periodic $\Gamma$ can be said to be helically periodic as well with respect to $\sigma$ for $\beta=0$.

b) Suppose the fundamental piece $\bar{\gamma}$ has at least two components. $\Gamma$ is said to be {\it rotationally periodic} (or {\it oddly periodic}) if the congruent copies of $\bar{\gamma}$ in $\Gamma$ are obtained by indefinitely extending $\bar{\gamma}$ through $180^\circ$-rotations about the lines connecting each pair of endpoints of $\bar{\gamma}$. On the other hand, $\Gamma$ is {\it reflectively periodic} (or {\it evenly periodic}) if the congruent copies of $\bar{\gamma}$ are obtained by indefinitely extending $\bar{\gamma}$ by the reflection across the planar faces of $\partial U$.

$\Gamma$ is complete because translations, screw motions, rotations and reflections are performed infinitely.
\end{definition}

\begin{theorem}\label{main}
Let $\Gamma\subset\mathbb R^3$ be the union of complete pairwise disjoint simple curves $\gamma_1,\gamma_2,\gamma_3,\ldots$ of infinite lengths. Suppose $\Gamma$ is periodic and its fundamental piece is a finite union of simple curves. Then there exists a periodic simply connected minimal surface $\Sigma$ spanning $\Gamma$. $\Sigma$ inherits the periodicity of $\Gamma$ and its fundamental region has least area among the fundamental regions of all the periodic simply connected surfaces spanning $\Gamma$.
\end{theorem}

\begin{proof}
Let's first prove the theorem when $\Gamma$ is helically periodic. We assume that $\Gamma$ is invariant under the $\sigma$ defined by (\ref{heli}). We may further assume that $\sigma$  maps the fundamental piece $\bar{\gamma}$ of $\Gamma$ to its adjoining piece, that is, $\bar{\gamma}$ is connected to $\sigma(\bar{\gamma})$ through their common endpoints. $\Gamma$ uniquely determines the angle $\beta>0$ of (\ref{heli}), which we call the {\it period} of $\Gamma$. $\hat{\Sigma}$ is a fundamental region of $\Sigma$ if and only if
$$\Sigma=\bigcup_{k\in\mathbb Z}\sigma^k(\hat{\Sigma})\,\,\,{\rm and}\,\,\, \hat{\Sigma}\cap\sigma(\hat{\Sigma})=\emptyset.$$
\begin{definition}
To each complete helically periodic curve $\Gamma$ we associate the class $\mathcal{C}_{a,\Gamma}$ of admissible maps $\varphi$ from the infinite strip $I_a:=[0,a]\times\mathbb R$ to $\mathbb R^3$ with the following properties:
\begin{enumerate}
\item $\varphi$ is a piecewise $C^{1}$ immersion in the interior of $I_a$ and is continuous in ${I_a}$;
\item $\varphi(x,y+k\beta)=\sigma^k(\varphi(x,y)),\,\, (x,y)\in I_a,\,\,k:{\rm integer},\,\, \beta:{\rm fixed}\,>0$;
\item $\varphi$ restricted to $\{0,a\}\times(0,\beta]$ is a monotone map onto a fundamental piece $\bar{\gamma}$ of $\Gamma$, i.e., $\bar{\gamma}$ is traversed once by $\varphi(\{0,a\}\times(0,\beta])$ although we allow arcs of $\{0,a\}\times(0,\beta]$ to map onto single points of $\bar{\gamma}$.
\end{enumerate}
To normalize $\mathcal{C}_{a,\Gamma}$ let's assume that $\varphi(0,0)=p$ for a fixed point $p$ of $\bar{\gamma}$.
$\varphi$ is said to be {\it invariant under the screw motion} $\sigma$ with {\it period} $\beta$ if $\varphi$ satisfies  property (2).
\end{definition}

Define the {\it area functional} $A$ on $\mathcal{C}_{a,\Gamma}$ by
$$A(\varphi)=\int\int_{[0,a]\times[0,\beta]}|\varphi_x\wedge\varphi_y|dx\,dy$$
and the {\it Dirichlet integral} $D(\varphi)$ of $\varphi\in\mathcal{C}_{a,\Gamma}$ by
$$D(\varphi)=\int\int_{[0,a]\times[0,\beta]}|\nabla\varphi|^2dx\,dy.$$
Since $$|\varphi_x\wedge\varphi_y|\leq\frac{1}{2}\left(|\varphi_x|^2+|\varphi_y|^2\right)$$
we have
\begin{equation}\label{A<D}
A(\varphi)\leq\frac{1}{2}D(\varphi),\,\,\varphi\in\mathcal{C}_{a,\Gamma}
\end{equation}
where equality holds if and only if $\varphi$ is almost conformal. In order to obtain the equality case, we need to prove the existence of periodic isothermal coordinates invariant under $\sigma$ on the surface $\varphi(I_a)$.
\begin{proposition}\label{prop1}
For any $\varphi\in\mathcal{C}_{a,\Gamma}$ there exists a periodic homeomorphism $H:I_a\rightarrow I_{\bar{b}}:=[0,\bar{b}]\times\mathbb R$ such that $H^{-1}$ has period $\beta$ and the reparametrized map $\varphi\circ H^{-1}:I_{\bar{b}}\rightarrow\varphi(I_a)$ is a conformal map in $\mathcal{C}_{\bar{b},\Gamma}$.
\end{proposition}
\begin{proof}
Let $N$ be the annulus obtained from $[0,a]\times[0,\beta]$ by identifying the two line segments $[0,a]\times\{0,\beta\}$. Let $g$ be the metric on $N$ which is pulled back by $\varphi$ from the metric of $\varphi(I_a)$. $g$ is well-defined since $\varphi$ is invariant under the screw motion $\sigma$ determined by $\Gamma$. Let's consider the Dirichlet problem on $(N,g)$ for constant $b>0$:
$$\Delta u=0,\,\,u=0\,\,{\rm on}\,\,\{0\}\times[0,\beta],\,\,u=b\,\,{\rm on}\,\,\{a\}\times[0,\beta].$$
There exists a unique solution $u=h_b$ to this problem. The harmonic function $h_b$ has a conjugate harmonic function $h_b^*$ which is multi-valued on $(N,g)$. But $h_b^*$ is well-defined on its universal cover $\widetilde{N}=I_a$. Let $\tau(b)>0$ be the period of $h_b^*$ on $N$. $\tau(b)$ is an increasing function which varies from 0 to $\infty$ as $b$ does so. Hence there exists $\bar{b}>0$ such that  $\tau(\bar{b})=\beta$. Note that  $h_{\bar{b}}$ can also be lifted to $h_{\bar{b}}$ on $I_a$. Then the map $H:I_a\rightarrow I_{\bar{b}}$ defined  by $H(q)=(h_{\bar{b}}(q),h_{\bar{b}}^*(q))$ is a periodic homeomorphism and yields a conformal map $\varphi\circ H^{-1}:I_{\bar{b}}\rightarrow\varphi(I_a)$. Note that $H^{-1}$ has period $\beta$ and $\varphi\circ H^{-1}$ is invariant under the screw motion $\sigma$. This completes the proof of the proposition.
\end{proof}
In order to prove the existence of an area-minimizing surface spanning $\Gamma$, let's define
$$a_\Gamma=\inf_{\varphi\in\mathcal{C}_{a,\Gamma} ,\,a>0}A(\varphi)\,\,\,\,{\rm and}\,\,\,\,d_\Gamma=\inf_{\varphi\in\mathcal{C}_{a,\Gamma},\,a>0}D(\varphi).$$
Then by (\ref{A<D}) and the existence of the isothermal coordinates we have
$$a_\Gamma=\frac{1}{2}\,d_\Gamma.$$
Therefore
$$D(\psi)=d_\Gamma\,\,{\rm for}\,\,{\rm some}\,\,\psi\in\mathcal{C}_{a,\Gamma} \,\, \Longleftrightarrow \,\,A(\psi)=a_\Gamma\,\,{\rm and}\,\, \psi\,\,{\rm is\,\,almost\,\,conformal}.$$
Thus, to solve the periodic Plateau problem it suffices to find $\bar{a}>0$ and a map $\psi\in\mathcal{C}_{\bar{a},\Gamma}$ which minimizes the Dirichlet integral $D(\varphi)$  on $[0,{a}]\times[0,\beta]$ among all $\varphi$ in $\mathcal{C}_{a,\Gamma}$ and all $a>0$.
First we shall fix $a>0$ and apply the periodic Dirichlet principle on $\mathcal{C}_{a,\Gamma}$ as follows.
\begin{lemma}\label{lemma1}
For each admissible map $\varphi$ in $\mathcal{C}_{a,\Gamma}$ there exists a unique harmonic admissible map $\psi\in\mathcal{C}_{a,\Gamma}$ with $\psi|_{\partial I_a}=\varphi|_{\partial I_a}$. Moreover, $D(\psi)\leq D(\varphi)$.
\end{lemma}
\begin{proof}
Let $x,y$ be the Euclidean coordinates of $\mathbb R^2$ and set $t=x+iy$. Define
$$f_1(t)=e^{\pi it/a}\,\,\,{\rm and}\,\,\,f_2(z)=\frac{iz+1}{z+i}.$$
Then $z=f_1(t)$ maps the infinite vertical strip $I_a$ one-to-one onto the upper half plane $\{{\rm Im}\,z\geq0\}\setminus\{0\}$ and $w=f_2(z)$ maps $\{{\rm Im}\,z\geq0\}\setminus\{0\}$ one-to-one onto the unit disk $\{|w|\leq1\}\setminus\{ i,-i\}$. Furthermore, we see that $f_2(f_1(\partial I_a))=\{|w|=1\}\setminus\{i,-i\}$. Let's consider the vector-valued Dirichlet problem for $u=(u_1,u_2,u_3)$ in $D:=\{w:|w|<1\}$:
\begin{equation}\label{bc0}
\Delta u=0\,\,{\rm in}\,\,D,\,\,\,\,\,\,u=\varphi\circ {f_1}^{-1}\circ f_2^{-1}\,\,{\rm on}\,\,\partial D,\,\,\,\,\,\,\varphi=(\varphi_1,\varphi_2,\varphi_3).
\end{equation}
Since $\varphi$ satisfies $\varphi(x,y+k\beta)=\sigma^k(\varphi(x,y))$ for the screw motion $\sigma$ defined by (\ref{heli}),
we see that $\varphi_1,\varphi_2$ are bounded and
\begin{equation}\label{y=x}
\varphi_3(x,y+k\beta)=\varphi_3(x,y)+ke.
\end{equation}

The Dirichlet problem (\ref{bc0}) has a unique bounded solution for $u_1,u_2$ because of the boundedness of $\varphi_1,\varphi_2$. Even though $\varphi_3$ is unbounded, by (\ref{y=x}) $\varphi_3-\frac{e}{\beta }y$ is bounded and periodic in $I_a$. So if the Dirichlet problem
\begin{equation}\label{bc1}
\Delta v=0 \,\,{\rm in}\,\,I_a,\,\,\,\,\,\,v=\varphi_3-\frac{e}{\beta }y\,\,{\rm on}\,\, \partial I_a
\end{equation}
has a bounded solution, it must be unique and periodic. To find its bounded solution, we convert it to a new Dirichlet problem on $D$:
\begin{equation}\label{bc2}
\Delta w =0\,\,{\rm in}\,\, D,\,\,\,\,\,\,w=(\varphi_3-\frac{e}{\beta}y)\circ {f_1}^{-1}\circ f_2^{-1}\,\,{\rm on}\,\,\partial D.
\end{equation}
The boundedness of $(\varphi_3-\frac{e}{\beta}y)\circ {f_1}^{-1}\circ f_2^{-1}$ gives the existence of a unique bounded solution $w=\tilde{h}_3$ to (\ref{bc2}). As $\frac{e}{\beta}y\circ f_1^{-1}\circ f_2^{-1}$ is harmonic in $D$, it is easy to see that   $u_3:=\tilde{h}_3+\frac{e}{\beta}y\circ f_1^{-1}\circ f_2^{-1}$ is the third component of a desired solution to (\ref{bc0}).

Pulling back $(u_1,u_2,u_3)$ by $f_2\circ f_1$ to $I_a$, one can obtain a harmonic map $\psi:I_a\rightarrow\mathbb R^3$ having the same boundary value as $\varphi$ on $\partial I_a$. We now show  that $\psi$ is invariant under the screw motion $\sigma$, in other words, $$\psi(x,y+\beta)=\sigma(\psi(x,y)).$$
Let $h_1,h_2,h_3:I_a\rightarrow\mathbb R$ be the harmonic components of $\psi$, that is,
$$\psi(x,y)=(h_1(x,y),h_2(x,y),h_3(x,y)).$$
(One easily sees that $h_3=\tilde{h}_3\circ f_2\circ f_1+\frac{e}{\beta}y$.)
Define
$$\psi_A(x,y)=\psi(x,y+\beta)\,\,{\rm and}\,\,\psi_B(x,y)=\sigma(\psi(x,y)).$$
Since $\tilde{h}_3\circ f_2\circ f_1$ is periodic with period $\beta$, we have
$$h_3(x,y+\beta)=h_3(x,y)+e.$$
So the third component of $\psi_A(x,y)$ equals that of $\psi_B(x,y)$.
On the other hand,
$$\psi_B(x,y)=(\cos\beta\, h_1(x,y)-\sin\beta\, h_2(x,y),\sin\beta \,h_1(x,y)+\cos\beta\, h_2(x,y),h_3(x,y)+e).$$
Hence $\psi_A,\psi_B$ are harmonic maps. As $h_1,h_2$ are bounded, so is $\psi_A-\psi_B$. Since $\sigma(\Gamma)=\Gamma$, $\psi_A-\psi_B$ vanishes on $\partial I_a$. Then $(\psi_A-\psi_B)\circ f_1^{-1}\circ f_2^{-1}$ is a bounded harmonic map vanishing on $\partial D$ and so $\psi_A-\psi_B\equiv0$.  Therefore $\psi$ is invariant under $\sigma$.
It follows that $\psi$ is a unique admissible harmonic map in $\mathcal{C}_{a,\Gamma}$ having the same boundary values as $\varphi$.

Set $\Phi=\varphi-\psi$. Then $\Phi$ is also invariant under $\sigma$ and hence
$$D(\varphi)=D(\Phi)+D(\psi)+2D(\Phi,\psi)$$
where
$$D(\Phi,\psi)=\int\int_{[0,a]\times[0,\beta]}\left(\langle\frac{\partial\Phi}{\partial x},\frac{\partial\psi}{\partial x}\rangle+\langle\frac{\partial\Phi}{\partial y},\frac{\partial\psi}{\partial y}\rangle\right)dxdy.$$
Green's identity implies that
$$D(\Phi,\psi)=\int_{\partial([0,a]\times[0,\beta])}\langle\Phi,\frac{\partial\psi}{\partial\nu}\rangle ds-\int\int_{[0,a]\times[0,\beta]}\langle\Phi,\Delta\psi\rangle dxdy,$$
where $\nu$ is the outward unit normal to $\partial([0,a]\times[0,\beta])$.
But
$$\Phi=0 \,\,{\rm on}\,\, \{0,a\}\times[0,\beta]\,\,\,\,\,\,{\rm and}\,\,\,\,\,\,\frac{\partial\psi}{\partial\nu}\big|_{[0,a]\times\{\beta\}}=-\frac{\partial\psi}{\partial\nu}\big|_{[0,a]\times\{0\}}$$
because of the invariance of $\psi$ under $\sigma$.
Hence  $D(\Phi,\psi)=0$. It then follows that
$$D(\psi)\leq D(\varphi),$$
which completes the proof of the lemma.
\end{proof}

Define
$$d_{a,\Gamma}={\rm inf}_{\varphi\in \mathcal{C}_{a,\Gamma}}D(\varphi).$$
We claim here that $d_{a,\Gamma}$ goes to infinity as $a\rightarrow\infty$ and as $a\rightarrow0$.
\begin{eqnarray*}
D(\varphi)&\geq&\int_0^a\int_0^\beta|\varphi_y|^2dydx=\int_0^a\int_0^\beta\sum_{i=1}^3\left(\frac{\partial\varphi_i}{\partial y}\right)^2dydx\\
&\geq&\frac{1}{\beta}\int_0^a\left(\int_0^\beta\frac{\partial\varphi_3}{\partial y}dy\right)^2dx=\frac{1}{\beta}\int_0^a(\varphi_3(x,\beta)-\varphi_3(x,0))^2dx\\
&=&\frac{ae^2}{\beta}.
\end{eqnarray*}
So $\lim_{a\rightarrow\infty}d_{a,\Gamma}=\infty$. On the other hand,
\begin{eqnarray*}
D(\varphi)&\geq&\int_0^\beta\int_0^a|\varphi_x|^2dxdy=\int_0^\beta\int_0^a\sum_{i=1}^3\left(\frac{\partial\varphi_i}{\partial x}\right)^2dxdy\\
&\geq&\frac{1}{a}\int_0^\beta\sum_{i=1}^3\left(\int_0^a\frac{\partial\varphi_i}{\partial x}dx\right)^2dy=\frac{1}{a}\int_0^\beta\sum_{i=1}^3(\varphi_i(a,y)-\varphi_i(0,y))^2dy\\
&\geq&\frac{\beta d^2}{a},
\end{eqnarray*}
where $d$ is the distance between the two components $\gamma_0$, $\gamma_1$ of $\Gamma$ which are written as $\gamma_0=\varphi(\{0\}\times\mathbb R), \gamma_1=\varphi(\{a\}\times\mathbb R)$.
Hence $\lim_{a\rightarrow0}d_{a,\Gamma}=\infty$ as well.

Therefore we can conclude that there exists a positive constant $\bar{a}$ such that
$$d_\Gamma=d_{\bar{a},\Gamma}.$$
To finish the proof of Theorem \ref{main} we need the following.
\begin{lemma}\label{lemma2}
Let $M$ be a constant $>d_{\Gamma}$. Then for any $a>0$ the family of functions
$$\mathcal{F}_a=\{\varphi|_{\partial I_{{a}}}:\varphi\in\mathcal{C}_{{a},\Gamma},\,\,D(\varphi)\leq M\}$$
is compact in the topology of uniform convergence.
\end{lemma}
\begin{proof}
For each $z\in \partial I_{{a}}$ and each $r>0$, define $C_r$ to be the intersection of $I_{{a}}$ with the circle of radius $r$ centered at $z$, and denote by $s$ the arc length parameter of $C_r$. Choose any $\varphi\in\mathcal{C}_{{{a}},\Gamma}$ with $D(\varphi)\leq M$. For $0<\delta<{\rm min}(1,{a}^2)$, consider the integral
$$K:=\int_\delta^{\sqrt{\delta}}\int_{C_r}|\varphi_s|^2ds\,dr\leq D(\varphi)\leq M.$$
One can see that
$$K=\int_\delta^{\sqrt{\delta}}f(r)\,d(\log r),\,\,\,\,f(r):=r\int_{C_r}|\varphi_s|^2ds.$$
By the mean value theorem there exists $\rho$ with $\delta\leq\rho\leq\sqrt{\delta}$ such that
$$K=f(\rho)\int_\delta^{\sqrt{\delta}}d(\log r)=\frac{1}{2}\,f(\rho)\log(\frac{1}{\delta}).$$
Hence
$$\int_{C_\rho}|\varphi_s|^2ds\leq \frac{2M}{\rho\log(\frac{1}{\delta})}.$$
Denote the length of the curve $\varphi(C_r)$ by $L(\varphi(C_r))$. Then $L(\varphi(C_\rho))=\int_{C_\rho}|\varphi_s|ds$ and from the Cauchy-Schwarz inequality it follows that
\begin{equation}\label{lc}
L(\varphi(C_\rho))^2\leq\frac{2\pi M}{\,\log(\frac{1}{\delta})}.
\end{equation}

Given a number $\varepsilon>0$, by the compactness of $\Gamma/\langle\sigma\rangle$ we see that there exists $d>0$ such that for any $p,p'$ in $\Gamma$ with $0<|pp'|<d$, the diameter of the bounded component of $\Gamma\setminus\{p,p'\}$ is smaller than $\varepsilon$. Choose $\delta<{\rm min}(1,{a}^2)$ such that ${\frac{2\pi M}{\log(\frac{1}{\delta})}}<d^2$. Then for any $z\in\partial I_{{a}}$, there exists  a number $\rho$ with $\delta<\rho<\sqrt{\delta}$ such that by (\ref{lc}), $L(\varphi(C_\rho))<d$.
Let $E_z$ be the interval in $\partial I_a$ between $z_1$ and $z_2$, the two endpoints of $C_\rho$. Then $|\varphi(z_1)\varphi(z_2)|<d$ and hence the diameter of $\varphi(E_z)$ is smaller than $\varepsilon$. Therefore for any $z,z'\in\partial I_{{a}}$ with $|z-z'|<{\delta}$ and $z$ being the center of $C_\rho$, we have $\varphi(z),\varphi(z')\in\varphi(E_z)$ and thus
$$|\varphi(z)-\varphi(z')|<\varepsilon.$$
Since $\delta$ was chosen independently of $z,z'$ and $\varphi$, we obtain the equicontinuity of $\mathcal{F}_a$.

In Douglas's solution for the existence of a conformal harmonic map $\varphi:D\rightarrow\mathbb R^n$ spanning a Jordan curve $\Gamma$, it was essential to prescribe $\varphi(z_i)=p_i$ for arbitrarily chosen points $z_1,z_2,z_3\in\partial D$ and $p_1,p_2,p_3\in\Gamma$. This was for the purpose of deriving the equicontinuity in a minimizing sequence. Fortunately we do not need this kind of prescription for our compact set $\Gamma/\langle\sigma\rangle$ as $\varphi(I_a)/\langle\sigma\rangle$ is not a disk. Yet we need to avoid an unwanted situation resulting from the disconnectedness of $\Gamma$: we have to show that $\varphi(\{a\}\times\mathbb R)$ does not drift away from $\varphi(\{0\}\times\mathbb R)$ (recall that $\varphi(0,0)$ is fixed). This can be done by deriving a length bound from a bound on $D(\varphi)$ as above.

For each $y\in [0,\beta]$ let $\ell_y$ denote the line segment $[0,a]\times \{y\}$. Choose $\varphi\in\mathcal{C}_{{{a}},\Gamma}$ and suppose $D(\varphi)\leq M$. Consider the integral
$$K:=\int_0^{{\beta}}\int_{\ell_y}|\varphi_x|^2dx\,dy\leq D(\varphi)\leq M.$$
Then
$$K=\int_0^{{\beta}}\tilde{f}(y)\,dy,\,\,\,\,\,\tilde{f}(y):=\int_{\ell_y}|\varphi_x|^2dx.$$
The mean value theorem implies that there exists $0<\bar{y}<\beta$ such that
$$K=\beta\,\tilde{f}(\bar{y})\leq M.$$
Hence
\begin{equation}\label{length}
L(\varphi(\ell_{\bar{y}}))^2=\left(\int_{\ell_{\bar{y}}}|\varphi_x|dx\right)^2\leq a\,\tilde{f}(\bar{y})\leq\frac{aM}{\beta}.
\end{equation}
We say that $\varphi(\{a\}\times\mathbb R)$ drifts away from $\varphi(\{0\}\times\mathbb R)$ if $\lim_{x\rightarrow a}|\varphi_3(x,y)|=\infty$ for some $0\leq y\leq\beta$. Therefore (\ref{length}) means that no drift occurs under $\varphi$ if $D(\varphi)$ is bounded, as claimed. Thus by Arzela's theorem the equicontinuity yields the compactness of $\mathcal{F}_a$. This completes the proof of Lemma \ref{lemma2}.
\end{proof}
Finally, let $\{\varphi_n\}$ be a minimizing sequence in $\mathcal{C}_{\bar{a},\Gamma}$, that is, $\lim_{n\rightarrow\infty}D(\varphi_n)=d_\Gamma$. From Lemma \ref{lemma2} it follows that there exists a subsequence $\{\varphi_{n_i}\}$ such that $\{\varphi_{n_i}|_{\partial I_{\bar{a}}}\}$ converges uniformly to $\bar{\varphi}|_{\partial I_{\bar{a}}}$ for some $\bar{\varphi}\in\mathcal{C}_{\bar{a},\Gamma}$. By Lemma \ref{lemma1} there exist harmonic maps $\psi_{i},\psi$ $\in\mathcal{C}_{\bar{a},\Gamma}$ such that
$$\psi_{i}|_{\partial I_{\bar{a}}}=\varphi_{n_i}|_{\partial I_{\bar{a}}},\,\,\,\,\,D(\psi_i)\leq D(\varphi_{n_i}),\,\,\,\,\,\psi|_{\partial I_{\bar{a}}}=\bar{\varphi}|_{\partial I_{\bar{a}}},\,\,\,\,\,\psi=\lim_{i\rightarrow \infty}\psi_{i}.$$
Then Harnack's principle gives
$$D(\psi)\leq\lim\inf_iD(\psi_{i})\leq d_\Gamma.$$
Consequently, $D(\psi)=d_\Gamma$ and so $\psi$ is almost conformal and harmonic. This completes the proof of Theorem \ref{main} when $\Gamma$ is helically periodic and therefore when it is translationally periodic as well. Since $\psi$ is periodically area minimizing in $\mathbb R^3$ it has no interior branch point (see \cite{G}).

Let $U$ be a convex polyhedral domain in $\mathbb R^3$ and $\bar{\gamma}:=\Gamma\cap U$ a fundamental piece of $\Gamma$. While a helically periodic $\Gamma$ has only two components, the fundamental piece $\bar{\gamma}$ of a rotationally (or reflectively) periodic $\Gamma$ may have more than two components. Let $\bar{\gamma}_1,\bar{\gamma}_2,\bar{\gamma}_3,\ldots,\bar{\gamma}_n$ be the components of $\bar{\gamma}$. For $i=1,\ldots,n$, let $p_i^1,p_i^2$ be the endpoints of $\bar{\gamma}_i$. Reordering $i=1,\ldots,n$ if necessary, we may assume that the line segment $\overline{p_i^1p_{i+1}^2}$ is in a planar face of $U$ connecting $\bar{\gamma}_i$ to $\bar{\gamma}_{i+1}$ and that $\Gamma_0:=\bar{\gamma}_1\cup\cdots\cup\bar{\gamma}_n\cup\overline{p_1^1p_2^2}\cup\overline{p_2^1p_3^2}\cup\cdots\cup
\overline{p_{n-1}^1p_n^2}\cup\overline{p_n^1p_1^2}$ is a Jordan curve in ${U}\cup\partial U$. There exists a Douglas solution $\Sigma_a$ spanning $\Gamma_0$. If $\Gamma$ is rotationally periodic, $\Gamma$ can be recaptured by indefinitely rotating $\bar{\gamma}_1\ldots,\bar{\gamma}_n$ $180^\circ$ around $\overline{p_1^1p_2^2},\ldots,\overline{p_n^1p_1^2}$ and around the corresponding line segments in the adjacent polyhedra. If we perform the same indefinite rotations on $\Sigma_a$ as we obtain $\Gamma$ from $\bar{\gamma}_1,\ldots,\bar{\gamma}_n$, then $\Sigma_a$ gives rise to a rotationally periodic minimal surface $\Sigma$ spanning $\Gamma$, as desired.

Suppose $\Gamma$ is reflectively periodic with fundamental piece $\bar{\gamma}$ in $U$. By the theorem of existence of minimizers for the free boundary problem (see Section 5.3 of \cite{DHKW}) there exists a minimal surface $\Sigma_b$ of least area in $U$ such that $\partial\Sigma_b\setminus\partial U=\bar{\gamma}$ and $\Sigma_b$ is perpendicular to $\partial U$ along $\partial\Sigma_b\cap\partial U$. Apply the same indefinite reflections to $\Sigma_b$ as we do to $\bar{\gamma}$ to get $\Gamma$. Then we can obtain a reflectively periodic minimal surface $\Sigma$ spanning $\Gamma$, as desired.
\end{proof}

\begin{remark}
(a) One can similarly consider a disjoint union $\Gamma$ of complete simple curves in the hyperbolic space $\mathbb H^3$ and in the sphere $\mathbb S^3$. $\Gamma$ can be compact in $\mathbb S^3$. One easily sees that Theorem \ref{main} still holds for $\Gamma$ in $\mathbb H^3$ and in $\mathbb S^3$.

(b) A periodic minimal surface (and its boundary) may be partly rotationally periodic and partly reflectively periodic.

(c) A helically periodic $\Sigma$ spanning $\Gamma$ with invariance group $\langle\sigma\rangle$ gives rise to the quotient surface $\Sigma/\langle\sigma\rangle$ and the quotient boundary $\Gamma/\langle\sigma\rangle$.
\end{remark}

\section{Uniqueness and Embeddedness}
Under what condition can $\Gamma$ guarantee the uniqueness and embeddedness of the periodic Plateau solution $\Sigma$? For the Douglas solution with Jordan curve $\Gamma$ Nitsche \cite{N} and Ekholm-White-Wienholtz \cite{EWW} proved the uniqueness and the embeddedness, respectively, if the total curvature of $\Gamma\leq4\pi$. But even before Douglas, Rad\'{o} \cite{R} showed that the Dirichlet solution of the minimal surface equation for any continuous boundary data over the boundary of a convex domain in $\mathbb R^2$ exists as a graph, which is obviously unique and embedded. In the same spirit, we have a partial answer for our periodic Plateau problem as follows.

\begin{theorem}\label{plateau}
Let $\gamma_0$ be the $x_3$-axis and $\gamma_1$ a complete connected curve winding around $\gamma_0$. Define $\Gamma=\gamma_0\cup\gamma_1$ and let $\tau$ be a vertical translation by $e$. If $\Gamma$ is translationally periodic with respect to $\tau$ and a fundamental piece of $\gamma_1$ admits a one-to-one orthogonal projection onto a convex closed curve in the $x_1x_2$-plane, then the translationally periodic minimal surface $\Sigma$ spanning $\Gamma$ has the following properties:
\begin{itemize}
\item[(a)] The Gaussian curvature of $\Sigma$ is negative at any point $p\in\gamma_0$;
\item[(b)] $\Sigma$ is embedded and  its fundamental region (not including $\gamma_0$) is a graph over its projection onto the $x_1x_2$-plane;
\item[(c)] $\Sigma$ is unique.
\end{itemize}
\end{theorem}
\begin{proof}
(a) $\gamma_0$ is parametrized by $x_3$. At any point $p(x_3)$ of $\gamma_0$, $\Sigma$ has a tangent {\it half plane} $Q_{p(x_3)}$. In a neighborhood of $p(x_3)$, $\Sigma$ is divided by $Q_{p(x_3)}$, like a half pie, into $m(\geq2)$ regions (see Figure 4). Define $\theta(x_3)$ to be the angle between $Q_{p(x_3)}$ and the positive $x_1$-axis. $\theta(x_3)$ is a well-defined analytic function satisfying $\theta(x_3+e)=\theta(x_3)+2\pi$. It is known (to be proved shortly) that
\begin{center}
\includegraphics[width=5.7in]{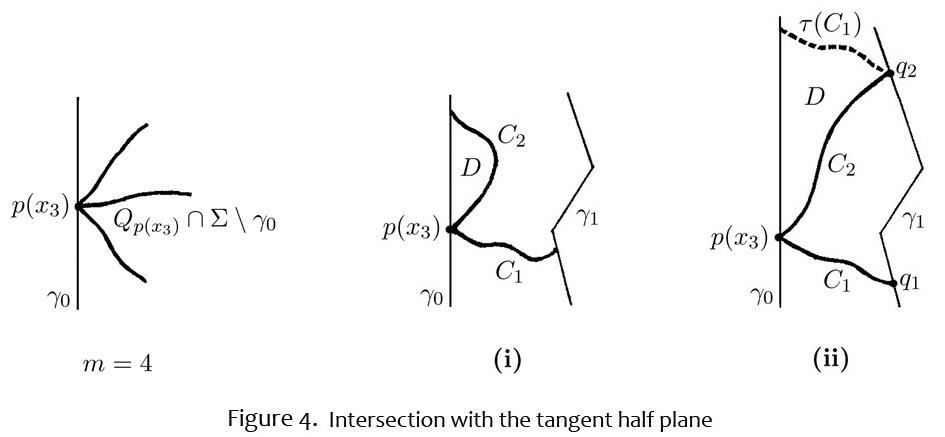}\\
\end{center}
\begin{equation}\label{equivalence}
m=2\,\,{\rm at}\,\,p(x_3) \,\,\, \Leftrightarrow\,\,\, K(x_3)<0 \,\,\, \Leftrightarrow\,\,\, \theta'(x_3)\neq0,
\end{equation}
where $K(x_3)$ is the Gaussian curvature of $\Sigma$ at $p(x_3)$.

We claim that $m\equiv2$ on $\gamma_0$. Suppose $m\geq3$ at $p(x_3)$ so that $Q_{p(x_3)}\cap\Sigma\setminus\gamma_0$ is the union of at least two analytic curves $C_1,C_2,...,C_k$ emanating from $p(x_3)$. Since
$Q_{p(x_3)}$ intersects $\gamma_1$, at least one of $C_1,C_2,...,C_k$ should reach $\gamma_1$. So we have two
possibilities: either {\bf (i)} only one of them, say $C_1$, reaches $\gamma_1$, or {\bf (ii)} two of them, say $C_1,C_2$, reach $\gamma_1$ (see Figure 4). In the first case, since $C_2$ is disjoint from $\gamma_1$ and
translationally periodic, it cannot be unbounded and should be in a fundamental region of $\Sigma$. Hence $C_2$ comes back to $\gamma_0$.
$C_2$ and $\gamma_{0}$ should then bound a domain $D\subset\Sigma$ with $\partial D\subset {Q_{p(x_3)}}$ as $\Sigma$ is simply connected. But this contradicts the maximum principle because $D$ has a point which attains the maximum distance from $Q_{p(x_3)}$. In case of {\bf (ii)}, set $C_1\cap\gamma_1=\{q_1\}$ and $C_2\cap\gamma_1=\{q_2\}$. Denote by $\pi$ the projection onto the $x_1x_2$-plane. Due to the convexity of $\pi(\gamma_1)$, $Q_{p(x_3)}$ intersects any fundamental piece of $\gamma_1$ only at one point. Therefore $\{q_1,q_2\}$ should be the boundary of a fundamental piece of $\gamma_1$. Hence $\tau(q_1)=q_2$, interchanging $q_1$ and $q_2$ if necessary. So the two curves $\tau(C_1)$ and $C_2$ meet at $q_2$. Then $\tau(C_1)$, $C_2$ and $\gamma_0$ bound a domain $D\subset\Sigma$. Again $\partial D$ is a subset of $Q_{p(x_3)}$, which is a contradiction to the maximum principle. Therefore $m\equiv2$ on $\gamma_0$, as claimed.

To give a proof of the equivalences (\ref{equivalence}), let's view $\Sigma$ in a neighborhood of $p\in\gamma_0$ as a graph over $Q_p$, the tangent half plane of $\Sigma$ at $p$. Introduce $x,y,z$ as the coordinates of $\mathbb R^3$ such that $z\equiv0$ on $Q_p$, $x\equiv0$ on $\gamma_0$ and $p=(0,0,0)$. Then $\Sigma$ is the graph of an analytic function  $z=f(x,y)$ and the lowest order term of its Taylor series is $f_m(x,y)=c_m\,{\rm Im}(x+iy)^m, m\geq2$,  when $m$ is an even integer and $f_m(x,y)=c_m\,{\rm Re }(x+iy)^m$ when $m$ is odd. It follows that $\Sigma$ is divided by $Q_p$ into $m$ regions in a neighborhood of $p$ and that $K(p)=0$ if $m\geq3$ and $K(p)<0$ if $m=2$, which is the first equivalence in (\ref{equivalence}).  Hence $K<0$ on $\gamma_0$ by the claim above and this proves (a). The second equivalence  follows from the expression for the Gaussian curvature in terms of the Weierstrass data on $\Sigma$, a 1-form $fdz$ and the Gauss map $g$:
\begin{equation}\label{K}
K=-\frac{16|g'|^2}{|f|^2(1+|g|^2)^4}.
\end{equation}

(b) First we show that $\Sigma\setminus\gamma_0$ has no vertical tangent plane.  Suppose not; let $q$ be an interior point of $\Sigma$ at which the tangent plane ${P}$ is vertical. Remember that $\pi(\gamma_1)$ is convex. Hence ${P}$ intersects $\gamma_1$ only at two points in its fundamental piece. ${P}\cap\Sigma$ is locally the union of at least four curves $C_1,\ldots,C_k, k\geq4,$ emanating from $q$, and two of them should reach $\gamma_1$. If we assume only four curves emanate from
$q$ in ${P}\cap\Sigma$, two of them will reach $\gamma_1$ and then either the remaining two will reach $\gamma_0$ or they will be connected to each other by the translation $\tau$ as in Figure 5: {\bf (i)} $C_1$,
$C_2$ will intersect $\gamma_1$ and  $C_3,C_4$ will intersect $\gamma_{0}$; {\bf (ii)} $C_1,C_2$ will intersect $\gamma_1$ and
$C_3,C_4$ will be disjoint from $\gamma_0\cup\gamma_{1}$ so that $C_4$ will be connected to $\tau(C_3)$. In case
of {\bf (i)}, $C_3\cup C_4\cup\gamma_{0}$ will bound a domain $D\subset\Sigma_{}$. But this is a contradiction to the maximum principle since $\partial D\subset {P}$.  In case of {\bf (ii)}, ${\gamma}_{0}$ is disjoint from ${P}$. Then ${\gamma}_{0}$ and ${P}\cap\Sigma_{}$ bound an infinite strip $S\subset\Sigma_{}$ lying on one side of ${P}$. Since $S/\langle\tau\rangle$ is compact, there exists a point $p_S\in S$ which has the maximum distance from ${P}$ among all points of $S$. $\gamma_0$ is a constant distance away from $P$ and the inward unit conormals to $\gamma_0$ on $\Sigma$ wind around it once in its fundamental piece. So there is a point in $\gamma_0$ at which the inward unit conormal to $\gamma_0$ points away from $P$. Then in
\begin{center}
\includegraphics[width=3.11in]{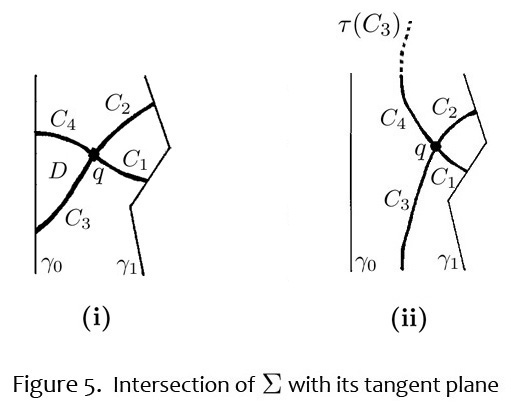}\\
\end{center}
that direction the distance from $P$ increases, hence $p_S$ is not a point of $\gamma_0$ but an interior point of $S$. However, this contradicts the maximum principle. Consequently, no tangent plane to $\Sigma$ can be vertical at any point of $\Sigma_0$. Even if ${P}\cap\Sigma$ consists of six curves or more, the same argument works.

We now show that the interior  of $\Sigma$ does not intersect $\gamma_0$. Let $\psi:[0,\bar{a}]\times\mathbb R\rightarrow\mathbb R^3$ be the periodically area minimizing conformal harmonic map such that $\psi([0,\bar{a}]\times\mathbb R)=\Sigma$, $\psi(\{0\}\times\mathbb R)=\gamma_0$ and $\psi(\{\bar{a}\}\times\mathbb R)=\gamma_1$. Suppose there exists an interior point $p\in(0,\bar{a})\times\mathbb R$ such that $\Sigma$ intersects $\gamma_0$ at $\psi(p)$. Define $f(q)=x_1(q)^2+x_2(q)^2$ for $q\in\Sigma$. Let $\mathcal{F}$ be the family of all arcs on $\Sigma$ connecting $\gamma_0$ to $\psi(p)$. Let's find a saddle point in $\Sigma$ for the function $f$. Define
$$A={\rm min}_{\alpha\in\mathcal{F}}\,{\rm max}_{q\in\alpha}f(q).$$
Clearly there exists a saddle point $q_0$ in $\Sigma$ such that $f(q_0)=A$. Suppose $A=0$. Then there is an arc $\tilde{\alpha}\subset[0,\bar{a}]\times\mathbb R$ connecting $\{0\}\times\mathbb R$ to $p$ such that $f\equiv0$ on $\psi(\tilde{\alpha})$. Since $\Sigma$ is periodically area minimizing, it has no interior branch point. Neither does $\Sigma$ have a boundary branch point on $\gamma_0$. Hence $\psi$ is an immersion on $[0,\bar{a})\times\mathbb R$. But $\psi$ maps $(\{0\}\times\mathbb R)\cup\tilde{\alpha}$ onto $\gamma_0$ if $f\equiv0$ on $\psi(\tilde{\alpha})$. This is not possible for the immersion $\psi$. Hence $A$ cannot be equal to $0$. Since $\nabla f=0$ at $q_0$, the tangent plane to $\Sigma$ at $q_0$ is parallel to $\gamma_0$ and hence it must be vertical. This is a contradiction. Therefore  the interior of $\Sigma$ does not intersect $\gamma_0$.

Henceforth we show that $\hat{\Sigma}\setminus\gamma_{0}$ is a graph over the $x_1x_2$-plane, where $\hat{\Sigma}$ is a fundamental region of $\Sigma$. By (a) we know that $m\equiv2$ on $\gamma_0$. Hence, given a vertical half plane $Q$ emanating from $\gamma_{0}$ and a suitably chosen fundamental region $\hat{\Sigma}$ of $\Sigma$, $\overline{Q\cap\hat{\Sigma}\setminus\gamma_0}$ is a single smooth curve joining $\gamma_{0}$ to $\gamma_1$. Since the interior of $\Sigma$ does not intersect $\gamma_0$, the projection map $\pi|_{Q\cap\hat{\Sigma}\setminus\gamma_0}$ is one-to-one near $\gamma_0$. As $\pi(\gamma_1)$ is convex and $\pi|_{\hat{\Sigma}\cap\gamma_1}$ is one-to-one, hence $\pi(\Sigma)$ lies inside $\pi(\gamma_1)$ and $\pi|_{\hat{\Sigma}}$ is one-to-one near $\gamma_1$. Suppose the curve $Q\cap\hat{\Sigma}\setminus\gamma_0$ contains a point $p$ at which its tangent line is vertical. Then the tangent plane to $\Sigma$ at $p$ is also vertical, which is a contradiction. Hence $Q\cap\hat{\Sigma}\setminus\gamma_0$ admits a one-to-one projection into $\pi(Q)$ for all $Q$. It follows that $\hat{\Sigma}\setminus{\gamma}_{0}$ is a 2-dimensional graph over $\pi(\hat{\Sigma}\setminus{\gamma}_{0})$.  Hence $\Sigma$ is embedded.

(c) Suppose there exist two periodic Plateau solutions $\Sigma_1,\Sigma_2$ spanning $\Gamma$. Assume that their fundamental regions $\hat{\Sigma}_1,\hat{\Sigma}_2$ are the graphs of analytic functions $f_1,f_2:D\subset x_1x_2$-plane $\rightarrow\mathbb R$, $D:=\pi({\Sigma}_1\setminus\gamma_0)=\pi({\Sigma}_2\setminus\gamma_0)$. Assume also that $f_1\geq f_2$. If there exists an interior point $p\in D$ such that $(f_1-f_2)(p)={\rm max}_{q\in D}(f_1-f_2)(q)$, we have a contradiction to the maximum principle. Hence $f_1-f_2$ has no interior maximum in $D$. Since $f_1-f_2\equiv0$ on $\pi(\gamma_1)$, it can have a maximum only at $\pi(\gamma_0)=(0,0)$. However, the maximum is attained {\it anglewise} as follows. Let $M={\rm sup}_{q\in D}(f_1-f_2)(q)$. Given a half plane $Q$ emanating from $\gamma_0$, let $M_Q={\rm  sup}_{q\in {Q\cap D}}(f_1-f_2)(q).$ Then $M={\rm max}_{Q}M_Q$. Hence there exists a half plane $Q_1$ emanating from $\gamma_0$ such that
$$M=\lim_{q\in \ell,\,q\rightarrow (0,0)}(f_1-f_2)(q), \,\,\,{\rm where}\,\,\ell=Q_1\cap D.$$
Then the parallel translate of $\Sigma_2$ by $M$, denoted as $\Sigma_2+M$, still contains $\gamma_0$ as $\Sigma_1$ does, lies on one side of $\Sigma_1$ (above $\Sigma_1$) and is tangent to $\Sigma_1$ at $x_3=q_1:=\lim_{q\in \ell,\,q\rightarrow(0,0)}f_1(q)$. Hence by the boundary maximum principle(boundary point lemma), $f_2+M\equiv f_1$, that is, $\Sigma_2+M=\Sigma_1$. Since $\Sigma_2+M$ spans $\Gamma+M$ and $\Sigma_1$ does $\Gamma$, $M$ must equal $0$ and thus follows the uniqueness of $\Sigma$.
\end{proof}

\section{Smyth's Theorem}
It was H.A. Schwarz \cite{Sc} who first constructed a triply periodic minimal surface in $\mathbb R^3$.  He started from  a regular tetrahedron, four edges of which forms a Jordan curve, which in turn generates a unique minimal disk. Schwarz found this surface using specific Weierstrass data. By applying his reflection principle he was able to extend the minimal disk across its linear boundary to obtain the $D$-surface. Then Schwarz introduced its conjugate surface, which he called the $P$-surface. This surface is embedded and triply periodic just like the $D$-surface. Moreover, part of it is a free boundary minimal surface in a cube.

It is interesting to notice that both $D$-surface and $P$-surface have fundamental regions which are free boundary minimal disks in two specific tetrahedra, respectively. However, this is not an accident; B. Smyth \cite{Sm} showed surprisingly that {\it any} tetrahedron contains as many as {\it three} free boundary minimal disks. In the remainder of the paper we are interested in applying Smyth's method to the periodic Plateau solutions. To do so, we shall first review Smyth's theorem in this section.

Given a tetrahedron $T$ in $\mathbb R^3$, let $F_1,F_2,F_3,F_4$ be its faces and $\nu_1,\nu_2,\nu_3,\nu_4$ the outward unit normals to the faces, respectively. Then any three of $\nu_1,\nu_2,\nu_3,\nu_4$ are linearly independent but all four of them are not. Hence there should exist positive numbers $c_1,c_2,c_3,c_4$ such that
\begin{equation}\label{lr}
c_1\nu_1+c_2\nu_2+c_3\nu_3+c_4\nu_4=0.
\end{equation}
In fact, we may assume
$$c_i={\rm Area}(F_i),\,\,\,i=1,2,3,4.$$
This is due to the divergence theorem applied on  the domain $T$ to the gradients of the harmonic functions $x_1,x_2,x_3$, the Euclidean coordinates of $\mathbb R^3$.

By (\ref{lr}) we see that there exists an oriented skew quadrilateral $\Gamma$ whose edges (as vectors) are $c_1\nu_1,c_2\nu_2,c_3\nu_3,c_4\nu_4$. The Jordan curve $\Gamma$ bounds a unique minimal disk $\Sigma$, which is the image $X(D)$ of a conformal harmonic map $X:=(x_1,x_2,x_3)$. It is well known that $x_1,x_2,x_3$ are also harmonic on $\Sigma$. Hence there exist their conjugate harmonic functions $x_1^*,x_2^*,x_3^*$ on $\Sigma$. Then $X^*:=(x_1^*,x_2^*,x_3^*)$ defines a conformal harmonic map from $D$ onto $\Sigma^*$ in $\mathbb R^3$. $X^*\circ X^{-1}:\Sigma\rightarrow\Sigma^*$ is a local isometry because of the Cauchy-Riemann equations. Therefore $\Sigma^*$ is a minimal surface locally isometric to $\Sigma$.

Let $y_i=b_i^1x_1+b_i^2x_2+b_i^3x_3$ be a linear function in $\mathbb R^3$ such that $\nabla y_i=c_i\nu_i,\,i=1,2,3,4$. Then $y_i$ is constant($=d_i$) on the face $F_i$. Suppose $u,v$ are isothermal coordinates on $\Sigma$ such that $v$ is constant along the edge $c_i\nu_i$. Then $dX(\frac{\partial}{\partial v})$ is perpendicular to the vector $c_i\nu_i$ on the edge $c_i\nu_i$. Hence $\frac{\partial y_i}{\partial v}=0$, and by Cauchy-Riemann $\frac{\partial y_i^*}{\partial u}=0$ on $c_i\nu_i$ as well, where $y_i^*:=b_i^1x_1^*+b_i^2x_2^*+b_i^3x_3^*$. Therefore $y_i^*$ is constant along the edge $c_i\nu_i$, meaning that the image $X^*(c_i\nu_i)$  lies on the plane $\{y_i^*=d_i^*\}$ for some constant $d_i^*$.

$dX(\frac{\partial}{\partial u})$ is parallel to $\nabla y_i$ along the edge $c_i\nu_i$. By Cauchy-Riemann, there exists a number $c(p)$ at $p\in c_i\nu_i$ such that
\begin{equation}\label{*}
c(p)(b_i^1,b_i^2,b_i^3)=dX(\frac{\partial}{\partial u})=dX^*(\frac{\partial}{\partial v}).
\end{equation}
Hence $dX^*(\frac{\partial}{\partial v})$ is parallel to $(b_i^1,b_i^2,b_i^3)$. Therefore $\Sigma^*$ is perpendicular to the plane $\{y_i^*=d_i^*\}$ along $X^*(c_i\nu_i)$. In conclusion, $\Sigma^*$ is  locally isometric to $\Sigma$ and is a free boundary minimal surface in a tetrahedron $T'$ which is similar to $T$. Thus $T$ contains a free boundary minimal surface which is a homothetic expansion  of $\Sigma^*$.

The skew quadrilateral $\Gamma$ depends on the order of $c_1\nu_1,c_2\nu_2,c_3\nu_3,c_4\nu_4$. Any edge of the four can be chosen to be the first in a quadrilateral. Hence there are $6=3!$ orderings of the four edges. But they can be paired off into three quadrilaterals with two opposite orientations. To be precise, for example, if the quadrilateral $\Gamma_1$ determined by four ordered vectors $(u,v,w,x)$ is reversely traversed, we get the quadrilateral $-\Gamma_1$ for the ordering $(-u,-x,-w,-v)$. Define an orthogonal map $\xi(p)=-p,\, p\in\mathbb R^3$, then $\xi(-\Gamma_1)$ is the quadrilateral determined by $(u,x,w,v)$. $\xi(-\Gamma_1)$ cannot be obtained from $\Gamma_1$ by a Euclidean motion. Even so, the two minimal disks spanning $\Gamma_1$ and $\xi(-\Gamma_1)$ are intrinsically isometric. Moreover, their conjugate surfaces are extrinsically isometric, i.e., they are identical modulo a Euclidean motion.  Therefore the six orderings of the four edges yield three geometrically distinct conjugate minimal disks which, if properly expanded, will be free boundary minimal surfaces in $T$.

\section{Free boundary minimal annulus}
By applying Smyth's theorem to a translationally periodic solution of the periodic Plateau problem we are going to construct four free boundary minimal annuli  in a tetrahedron.

\begin{theorem}\label{fb}
Let $T$ be a tetrahedron with faces $F_1,F_2,F_3,F_4$ in $\mathbb R^3$ and let $\pi_i$ be the orthogonal projection onto the plane $P_i$ containing $F_i, i=1,2,3,4$.
\begin{itemize}
\item[(a)] If every dihedral angle of $T$ is $\leq90^\circ$, there exist four free boundary minimal annuli $A_1,A_2,A_3,A_4$ in $ T$.
\item[(b)] If at least one dihedral angle of $T$ is $>90^\circ$, there exist four minimal annuli $A_1,A_2,A_3$, $A_4$ which are perpendicular to $\cup_{j=1}^4 P_j$ along $\partial A_i$. Part of $A_i$ may lie outside $T$ if a dihedral angle is nearer to $180^\circ$. {\rm (}See {\rm Figure 6, right.)} Near $\partial A_i$, however, $A_i$ lies in the same side of $P_j$ as $T$ does. Moreover, $\partial A_i$ equals $\Gamma_i^1\cup\Gamma_i^2$, where $\Gamma_i^1$ is a closed convex curve in $P_i$ and $\Gamma_i^2$ is a closed, piecewise planar curve in $P_j\cup P_k\cup P_l$ with $ \{i,j,k,l\}=\{1,2,3,4\}$.
\item[(c)] If the three dihedral angles along $\partial F_i$ are $\leq90^\circ$, then $A_i$ lies inside $ T$. $\Gamma_i^1$ is a closed convex curve in $F_i$ and $\Gamma_i^2$ is a closed, piecewise planar curve in $\partial T\setminus F_i$. {\rm (}See {\rm Figure 6, left.)}
\item[(d)] Each planar curve in $\Gamma_i^2$ is convex and is perpendicular to the lines containing the edges of $T$ at its end points.
\item[(e)] $A_i$ is an embedded graph over $\pi_i(A_i)$.
\end{itemize}
\end{theorem}
\begin{center}
\includegraphics[width=5.5in]{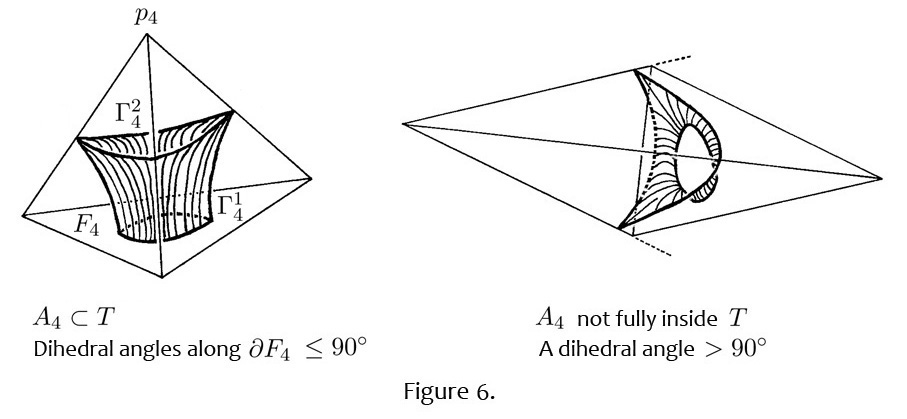}\\
\end{center}
\begin{proof}
As in the preceding section, $\nu_i$ denotes the outward unit normal to $F_i$. Again, there are positive constants $c_i={\rm Area}(F_i)$ such that $c_1\nu_1+c_2\nu_2+c_3\nu_3+c_4\nu_4=0$. Assume that $\nu_4$ is parallel to the $x_3$-axis so that $F_4$ is contained in the $x_1x_2$-plane. Denote the $x_1x_2$-plane by $P_4$ and recall that $\pi_4$ denotes the orthogonal projection onto  $P_4$. Since
$$\pi_4(c_1\nu_1)+\pi_4(c_2\nu_2)+\pi_4(c_3\nu_3)=0,$$
$\pi_4(c_1\nu_1),\pi_4(c_2\nu_2),\pi_4(c_3\nu_3)$ determine the boundary of a triangle $\Delta_4\subset P_4$, that is, $\pi_4(c_i\nu_i)$ is the $i$th oriented edge of $\Delta_4$, $i=1,2,3$. $\pi_4(c_i\nu_i)$ is perpendicular to the boundary edge $F_i\cap F_4$ of $F_4$. Also $\pi_4(c_i\nu_i)$ is perpendicular to the corresponding edge of $J(\Delta_4)$, where $J$ denotes the counterclockwise $90^\circ$ rotation on $P_4$. Therefore $\Delta_4$ is similar to $F_4$.

Choose a point $q$ from the interior $\check{\Delta}_4$ of $\Delta_4$ and let $\bar{\gamma}_q$ be the vertical line segment starting from $q$ and corresponding to (i.e., having the same length and direction as) $-c_4\nu_4$. Let $\bar{\gamma}_1$ be a connected piecewise linear open curve starting from a vertex of $\Delta_4$ that is the starting point of the oriented edge $\pi_4(c_1\nu_1)$ such that $\bar{\gamma}_1$ is the union of the three oriented line segments corresponding to the ordered vectors $c_1\nu_1,c_2\nu_2,c_3\nu_3$. Then $\pi_4(\bar{\gamma}_1)=\partial\Delta_4$. Also the endpoints of $\bar{\gamma}_1$ and $\bar{\gamma}_q$ are in ${\Delta}_4$ and in its parallel translate. One can extend $\bar{\gamma}_q\cup\bar{\gamma}_1$ into a complete translationally periodic curve $\Gamma_q:=\gamma_q\cup\gamma_1$ such that $\bar{\gamma}_q\cup\bar{\gamma}_1,\bar{\gamma}_q,\bar{\gamma}_1$ become  fundamental pieces of $\Gamma_q,\gamma_q,\gamma_1$, respectively. By Theorem \ref{main} and Theorem \ref{plateau} there uniquely exists a simply connected minimal surface $\Sigma_q$ spanning $\Gamma_q$. $\Sigma_q$ has the same translational periodicity as $\Gamma_q$ does. (See Figure 7.)

Let $\Sigma_q^*$ be the conjugate minimal surface of $\Sigma_q$ and denote by $Y_q^*=X_q^*\circ X_q^{-1}$ the local isometry from $\Sigma_q$ to $\Sigma_q^*$. By Smyth's arguments in the preceding section we see that the image $Y_q^*(c_i\nu_i)$ of the edge $c_i\nu_i$ is in a plane parallel to the face $F_i$. More precisely, $Y_q^*(c_i\nu_i)$ lies in the plane $\{y_i^*=d_i^*\}$, where $\nabla y_i^*=c_i\nu_i$.  However, $Y_q^*(\bar{\gamma}_q)$ is not closed in general because $Y_q^*$ may have nonzero period along $\bar{\gamma}_q$. But
\begin{center}
\includegraphics[width=5.5in]{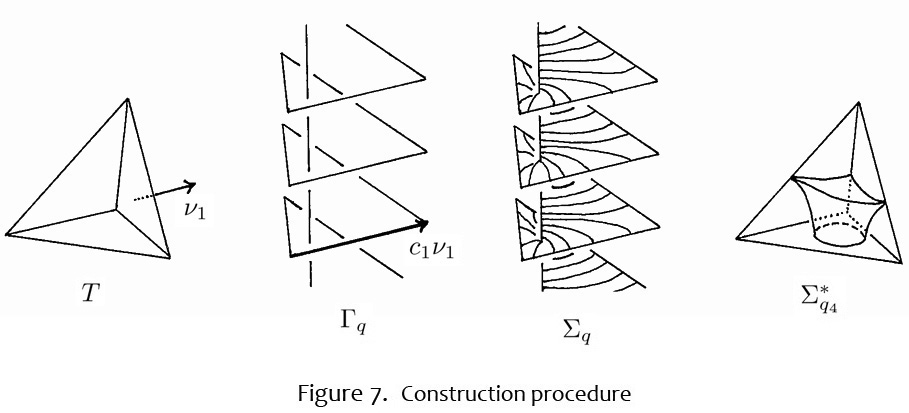}\\
\end{center}
note that by Cauchy-Riemann the period of $Y_q^*$ along $Y_q^*(\bar{\gamma}_q)$ equals the flux of $\Sigma_q$ along $\bar{\gamma}_q$. Therefore in order to make $\Sigma_q^*$ a well-defined compact minimal annulus, we need to find a suitable point $q$ in $\check{\Delta}_4$ for which the flux of $\Sigma_q$ along $\bar{\gamma}_q$ becomes the zero vector. Note here that the flux of $\Sigma_q$ along $\bar{\gamma}_1$ vanishes if and only if the flux of $\Sigma_q$ along $\bar{\gamma}_q$ does.

Let $n(p)$ be the inward unit conormal  to $\bar{\gamma}_q$ on $\Sigma_q$ at $p\in\bar{\gamma}_q$ and define
$$f(q)=\int_{p\in\bar{\gamma}_q}n(p).$$
Then $f(q)$ is the flux of $\Sigma_q$ along $\bar{\gamma}_q$ and $f$ is a map from the interior $\check{\Delta}_4$ to the set $N$ of vectors parallel to the plane $P_4$. $f$ is a smooth map and can be extended continuously to the closed triangle $\Delta_4$. Let ${\Delta}_4\times\mathbb R$ be the vertical prism over ${\Delta}_4$. Obviously $\Sigma_q$ lies inside ${\Delta}_4\times\mathbb R$. Since $\bar{\gamma}_1$ winds around $\bar{\gamma}_q$ once, so does $n(p)$ as $p$ moves along $\bar{\gamma}_q$. But as $q$ approaches a point $\tilde{q}\in\partial\Delta_4$, $\Gamma_q$ converges to a complete translationally periodic curve ${\Gamma}_{\tilde{q}}:={\gamma}_{\tilde{q}}\cup{\gamma}_1$ of which $\bar{\gamma}_{\tilde{q}}\cup\bar{\gamma}_1$ is a fundamental piece. Let $\tau$ be the translation defined by $\tau(\bar{p})=\bar{p}-c_4\nu_4,\, \bar{p}\in\mathbb R^3$. Since $\bar{\gamma}_{\tilde{q}}$ intersects $\bar{\gamma}_1$, ${\Gamma}_{\tilde{q}}$ is a periodic union of Jordan curves, or more precisely, ${\Gamma}_{\tilde{q}}=\cup_n\tau^n({{\gamma}}_{1\tilde{q}})$, where ${\gamma}_{1\tilde{q}}$ is a Jordan curve  which is a subset of $(\bar{\gamma}_{\tilde{q}}\cup\bar{\gamma}_1)\cup\tau(\bar{\gamma}_{\tilde{q}}\cup\bar{\gamma}_1)$. ${\gamma}_{1\tilde{q}}$ consists of five (or four if $\bar{\gamma}_{\tilde{q}}$ passes through a vertex of $\bar{\gamma}_1$) line segments. It is known that the total curvature of ${\gamma}_{1\tilde{q}}$ equals the length of its tangent indicatrix $T_{1\tilde{q}}$. $T_{1\tilde{q}}$ is comprised of (i) a geodesic triangle and a geodesic with multiplicity 2 in case ${\gamma}_{1\tilde{q}}$ consists of five line segments or (ii) four geodesics connecting the four points in $\mathbb S^2$ that correspond to $\nu_1,\nu_2,\nu_3,\nu_4$.  Since the length of a geodesic triangle is less than $2\pi$ and the length of a geodesic is less than $\pi$, the total length of $T_{1\tilde{q}}$ is smaller than $4\pi$ in either case. Thus by \cite{N} there exists a unique minimal disk spanning ${\gamma}_{1\tilde{q}}$ . As a matter of fact, we can easily extend the proof of Theorem \ref{plateau} (c) to the limiting case where $\bar{\gamma}_1$ intersects $\bar{\gamma}_0$. So we can see that ${\gamma}_{1\tilde{q}}$ bounds a unique minimal surface $\hat{\Sigma}_{\tilde{q}}\subset\Delta_4\times\mathbb R$ regardless of its topology. As $q\rightarrow\tilde{q}\in\partial\Delta_4$, a fundamental region of $\Sigma_q$ converges to $\hat{\Sigma}_{\tilde{q}}$.  Hence, by continuity of the extended map $f:\Delta_4\rightarrow N$, $f(q)$ converges to $f(\tilde{q})=\int_{p\in{\bar{\gamma}}_{\tilde{q}}}n(p)$ which is the flux of $\hat{\Sigma}_{\tilde{q}}$ along ${\bar{\gamma}}_{\tilde{q}}\subset\partial\Delta_4\times\mathbb R$. Therefore, as $n(p)$ points into the interior of $\Delta_4$ at any $p\in{\bar{\gamma}_{\tilde{q}}}$, $f(\tilde{q})$ is a nonzero horizontal vector pointing toward the interior of $\Delta_4$.

Now we are ready to show that there is a point $q$ in the interior $\check{\Delta}_4$ at which the flux $f(q)$ vanishes. Suppose $f(q)\neq0$ for all $q\in\check{\Delta}_4$ and define a map $\tilde{f}:{\Delta}_4\rightarrow \mathbb S^1$ by
$$\tilde{f}(q)=\frac{f(q)}{|f(q)|}.$$
Then $\tilde{f}$ is continuous and $\tilde{f}\big|_{\partial\Delta_4}$ has winding number $1$ because the nonzero horizontal vector $f(\tilde{q})$ points toward the interior $\check{\Delta}_4$ at any $\tilde{q}\in\partial\Delta_4$. But this is a contradiction since the induced homomorphism $\tilde{f}_*:\pi_1(\Delta_4)\rightarrow\pi_1(\mathbb S^1)$ must then be surjective. Therefore there should exist $q_4\in\check{\Delta}_4$, and a minimal surface $\Sigma_{q_4}$ which has zero flux $f(q_4)=0$ along $\bar{\gamma}_{q_4}$. Thus the conjugate surface $\Sigma_{q_4}^*$ is a well-defined minimal annulus. (See Figure 7.)

It remains to show that a homothetic expansion of $\Sigma_{q_4}^*$ is in $T$ and perpendicular to $\partial T$ along its boundary. According to the arguments of Smyth's theorem, there exist constants $d_1^*,d_2^*,d_3^*,d_4^*$ such that the curve $Y^*_{q_4}(c_i\nu_i)$ is in the plane $\{y_i^*=d_i^*\}$ and $\Sigma_{q_4}^*$ is perpendicular to that  plane along $Y^*_{q_4}(c_i\nu_i)$. Moreover, the outward unit conormal to $Y_{q_4}^*(c_i\nu_i)$ on $\Sigma_{q_4}^*$ is $\nu_i$ and hence near  $Y_{q_4}^*(c_i\nu_i)$, $\Sigma_{q_4}^*$ lies in the same side of the plane $\{y_i^*=d_i^*\}$ as $T'$ does. Remember that the four planes  $\cup_{i=1}^4\{y_i=d_i\}$ enclose the tetrahedron $T$ and $\cup_{i=1}^4\{y_i^*=d_i^*\}$ enclose the tetrahedron $T'$. Since $y_i=b_i^1x_1+b_i^2x_2+b_i^3x_3$ and $y_i^*=b_i^1x_1^*+b_i^2x_2^*+b_i^3x_3^*$, $T'$ is similar to $T$. As $\nu_4$ is assumed to be parallel to the $x_3$ axis, $y_4^*=b_4^*x_3^*$. 

Obviously a homothetic expansion of $\Sigma^*_{q_4}$ will give a minimal annulus $A_4$ which is perpendicular to $\cup_{i=1}^4\{y_i=d_i\}$ along $\partial A_4$. Working with a new plane $P_j$ containing $F_j, j=1,2,3,$ instead of $F_4$ and using the triangles $\Delta_j \subset P_j$, obtained from the relation for the projection $\pi_j$ into $P_j$:
$$\left(\sum_{i=1}^4\pi(c_i\nu_i)\right)-\pi_j(c_j\nu_j)=0,\,\,\,j=1,2\,\,{\rm or}\,\,3,$$
one can similarly find minimal annuli $A_1,A_2,A_3$ which are homothetic expansions of $\Sigma_{q_j}^*$ for some $q_j\in\Delta_j,j=1,2,3$. This proves (b) except for the convexity of the closed curve.

Let's denote by $F_j'$ the face of $T'$ which is similar to the face $F_j$ of $T$, $j=1,2,3,4$. Is it true that $\partial\Sigma_{q_j}^*\subset\partial T'$? Here we have to be careful because $Y^*_{q_j}(\bar{\gamma}_{q_j})$ and $Y^*_{q_j}(\bar{\gamma}_{1})$ are {\it disconnected}. (Notice that $\partial\Sigma^*$ is connected in Smyth's case.) Consequently, for $j=4$, $Y^*_{q_4}(\bar{\gamma}_{1})$ is not necessarily a subset of $\partial T'\setminus\{y_4^*=d_4^*\}$ and it may intersect the plane $\{y_4^*=d_4^*\}(=\{x_3^*=0\})$ as in Figure 6, right. To get some information about the location of $\partial\Sigma_{q_4}^*$, let's first assume that (d) and (e) are true. Since near  $Y_{q_4}^*(c_i\nu_i), i=1,2,3,$ $\Sigma_{q_4}^*$ lies in the same side of the plane $\{y_i^*=d_i^*\}$ as $T'$ does and since $Y_{q_4}^*(c_i\nu_i)$ are convex and are perpendicular on their endpoints to the three lines containing the edges $F_1'\cap F_2'$, $F_2'\cap F_3'$, $F_3'\cap F_1'$, respectively, one can conclude that {\it (i)} $Y_{q_4}^*(\bar{\gamma}_1)$ lies in the tangent cone $TC_{p_4'}(\partial T')$ of $\partial T'$ at $p_4'$, the vertex of $T'$ opposite $F_4'$. As $\Sigma_{q_4}^*$ is a graph over $\pi_4(\Sigma_{q_4}^*)$, {\it (ii)} $Y_{q_4}^*(\bar{\gamma}_{q_4})$ is surrounded by $\pi_4(Y_{q_4}^*(\bar{\gamma}_1))$ in the plane $\{y_4^*=d_4^*\}$.

Now let's prove a lemma which is more general than (c). If the dihedral angles along $\partial F_4$ are $\leq90^\circ$, the unit normals $\nu_1,\nu_2,\nu_3$ are pointing upward and  $\bar{\gamma}_1$ goes upward. So one can consider the following generalization.
\begin{lemma}\label{convexity}
Let $\Gamma=\gamma_0\cup\gamma_1$ be a translationally periodic curve and $\gamma_0$ the $x_3$-axis. Assume that $\Sigma_\Gamma$ is a translationally periodic Plateau solution spanning $\Gamma$. If $x_3$ is a nondecreasing function on $\gamma_1$, then the boundary component of $\Sigma_\Gamma^*$ corresponding to $\gamma_0$ is in the $x_1^*x_2^*$-plane and  $\Sigma_\Gamma^*$ is on and above the $x_1^*x_2^*$-plane.
\end{lemma}
\begin{proof}
$\Sigma_\Gamma$ has no horizontal tangent plane $T_p\Sigma_\Gamma$ at any interior point $p\in\Sigma_\Gamma$. This can be verified as follows. Every horizontal plane $\{x_3=h\}$ intersects $\Gamma$ either at two points only or at infinitely many points (the second case occurs when $\{x_3=h\}\cap\gamma_1$ is a curve of positive length). If $T_p\Sigma_\Gamma=\{x_3=h\}$, then $\{x_3=h\}\cap\Sigma_\Gamma$ is the set of at least four curves emanating from $p$. But then three of them intersect $\gamma_1$ and hence there exists a domain $D\subset\Sigma_\Gamma$ with $\partial D\subset\{x_3=h\}$, which contradicts the maximum principle. Hence $\{x_3=h\}$ is transversal to $\Sigma_\Gamma$ for every $h$ and therefore $x_3^*$ is an increasing function on every horizontal section $\{x_3=h\}\cap\Sigma_\Gamma$. Since $x_3^*=0$ on $\gamma_0$, $x_3^*$ must be nonnegative on $\Sigma_\Gamma^*$.
\end{proof}

If the dihedral angles along $\partial F_4$ are $\leq90^\circ$, then by the above lemma $Y_{q_4}^*(\bar{\gamma}_1)\subset \partial T'\setminus F_4'$. By (e), which will be proved independently, $Y_{q_4}^*(\bar{\gamma}_{q_4})$ is surrounded by $\pi_4(Y_{q_4}^*(\bar{\gamma}_1))$ and hence  $Y_{q_4}^*(\bar{\gamma}_{q_4})$ lies inside $F_4'$. This proves (c) (except for convexity) and (a) as well.

We now derive the convexity of $\partial\Sigma_{q_4}^*$ as follows. Henceforth our proof will be independent of (a), (b), (c). It should be mentioned that $\Sigma_{q_4}^*$ has been constructed independently of (d) and (e). Let $Q$ be a vertical half plane emanating from $\bar{\gamma}_{q_4}$, that is, $\partial Q\supset\bar{\gamma}_{q_4}$. Then $Q\cap\bar{\gamma}_1$ is a single point unless $Q$ contains the two boundary points of $\bar{\gamma}_1$. Let $q$ be a point of $\bar{\gamma}_{q_4}$ which is the end point of $Q\cap(\Sigma_{q_4}\setminus\bar{\gamma}_{q_4})$. Here we claim that in a neighborhood $U$ of $q$, $C:=U\cap Q\cap(\Sigma_{q_4}\setminus\bar{\gamma}_{q_4})$ is a  single curve emanating from $q$. If not, $U\cap Q\cap(\Sigma_{q_4}\setminus\bar{\gamma}_{q_4})$ is the union of at least two curves $C_1,C_2,\ldots$ emanating from $q$. These curves can be extended all the way up to $\bar{\gamma}_{q_4}\cup\bar{\gamma}_1$. In case $Q\cap\partial\bar{\gamma}_1=\emptyset$, $Q\cap\bar{\gamma}_1$ is a single point,  then only one of $C_1,C_2,\ldots$, say $C_1$, can reach the point $Q\cap\bar{\gamma}_1$ and  $C_2$ can only reach $\bar{\gamma}_{q_4}$. Since $\Sigma_{q_4}$ is simply connected,  $C_2$ and $\bar{\gamma}_{q_4}$ bound a domain $D\subset\Sigma_{q_4}$ with $\partial D\subset Q$. This contradicts the maximum principle. In case $Q$ intersects $\bar{\gamma}_1$ at its boundary points $p_1,p_2$, there exist two curves, say $C_1,C_2\subset Q\cap \Sigma_{q_4}$ emanating from $q$, such that $p_1\in C_1$ and $p_2\in C_2$. Remember that $\bar{\gamma}_{q_4}\cup\bar{\gamma}_1$ is a fundamental piece of $\Gamma_{q_4}$ which is translationally periodic under the vertical translation $\tau$ by $-c_4\nu_4$. Hence $\tau(p_1)=p_2$ and therefore the two distinct curves $\tau(C_1),C_2\subset Q\cap\Sigma_{q_4}$ emanate from $p_2$. But this is not possible since in a neighborhood of $p_2$, $Q\cap\Sigma_{q_4}$ is a single curve emanating from $p_2$. Hence the claim follows.

Note that $\log g=i\,{\rm arg}\,g$ on the straight line $\gamma_{q_4}$ containing $\bar{\gamma}_{q_4}$ because $|g|\equiv1$ there. If $(d/dx_3){\rm arg}\,g=0$ at a point $q\in\gamma_{q_4}$ ($x_3$: the parameter of ${\gamma}_{q_4}$), then for the vertical half plane $Q$ tangent to $\Sigma_{q_4}$ at $q$, $Q\cap(\Sigma_{q_4}\setminus{\gamma}_{q_4})$ will be the union of at least two curves emanating from $q$, contradicting the claim. Hence $g'\neq0$  on ${\gamma}_{q_4}$. Therefore $g'\neq0$ on $\Sigma_{q_4}^*\cap \{y_4^*=d_4^*\}=Y_{q_4}^*(\gamma_{q_4})$ as well and so $\Sigma_{q_4}^*\cap \{y_4^*=d_4^*\}$ is convex. Similarly, let $Q_j$ be a half plane emanating from the line segment $L$ in $\bar{\gamma}_1$ corresponding to $c_j\nu_j$, $j=1,2,3$. Being nonvertical, $Q_j$ intersects ${\gamma}_{q_4}$ only at one point. Hence $Q_j\cap(\Sigma_{q_4}\setminus L)$ is a single curve joining a point $p\in L$ to $Q_j\cap{\gamma}_{q_4}$ and $p$ is a tangent point of $Q_j$ and $\Sigma_{q_4}$. If we rotate $\Sigma_{q_4}$ in such a way that $|g|\equiv1$ on $L$, we can conclude $g'(p)\neq0$ in the same way as above, as long as $p$ is an interior point of $L$. On the other hand, $g'=0$ at the boundary of $L$ because the interior angle at the boundary of $L$ is $<\pi$. Note that any interior point of $L$ can be a tangent point of $Q_j$ and $\Sigma_{q_4}$ for some $Q_j$ emanating from $L$ and that $Q_j$ intersects $\gamma_{q_4}$ at one point only. Therefore $g'\neq0$ in the interior of $L\subset\Sigma_{q_4}$ and hence $g'\neq0$ in the interior of $\Sigma_{q_4}^*\cap \{y_j^*=d_j^*\}=Y^*(L)$. Thus $\Sigma_{q_4}^*\cap \{y_j^*=d_j^*\}$ is convex, $j=1,2,3$.
Since $\Sigma_{q_4}^*$ is perpendicular to $\{y_i^*=d_i^*\}$ and to $\{y_j^*=d_j^*\}$ at $p=\Sigma_{q_4}^*\cap \{y_i^*=d_i^*\}\cap \{y_j^*=d_j^*\}, 1\leq i\neq j\leq3$, so is $\partial \Sigma_{q_4}^*$ to the edge $\{y_i^*=d_i^*\}\cap \{y_j^*=d_j^*\}$ at $p$. This proves (d).

Remark that $Q\cap\bar{\gamma}_1$ being a single point is the key to the convexity of $\Sigma_{q_4}^*\cap \{y_4^*=d_4^*\}$. Therefore one can easily prove the following generalization which is dual to Lemma \ref{convexity}.

\begin{lemma}
Let $\Gamma=\gamma_0\cup\gamma_1$ be a translationally periodic curve and $\gamma_0$ the $x_3$-axis. Assume that $\Sigma_\Gamma$ is a translationally periodic Plateau solution spanning $\Gamma$ and that its conjugate surface $\Sigma_\Gamma^*$ is a well-defined minimal annulus. If a fundamental piece $\bar{\gamma}_1$ of $\gamma_1$ has a one-to-one projection into the $x_1x_2$-plane $\{x_3=0\}$, then the closed curve $\Sigma_\Gamma^*\cap\{x_3^*=0\}$ is convex.
\end{lemma}

Finally let's prove (e). Theorem \ref{plateau} (b) implies that $\hat{\Sigma}_{q_4}\setminus\gamma_{q_4}$ is a graph over $\pi_4(\Sigma_{q_4}\setminus{\gamma}_{q_4})$. The two boundary curves  $\partial\hat{\Sigma}_{q_4}\setminus(\gamma_{q_4}\cup\gamma_1)$ are the parallel translates of one another. Therefore $\Sigma_{q_4}$ is embedded. Now we are going to use Krust's argument (see Section 3.3 of \cite{DHKW}) to prove that $\Sigma_{q_4}^*$ is also a graph. Let $X=(x_1,x_2,x_3)$ be the immersion of $[0,a]\times[0,\beta]$ into $\Sigma_{q_4}$ and $X^*=(x_1^*,x_2^*,x_3^*)$ the immersion: $[0,a]\times[0,\beta]\rightarrow\Sigma_{q_4}^*$. We can write the orthogonal projections of $X$ and $X^*$ into the horizontal plane as respectively
$$w(z):=x_1(z)+ix_2(z),\,\,w^*(z):=x_1^*(z)+ix_2^*(z),\,\,z=x+iy,\,\,(x,y)\in[0,a]\times[0,\beta].$$
Then $w$ is a map from $[0,a]\times[0,\beta]$ onto the triangle $\Delta_4$. Given two distinct points $z_1,z_2\in(0,a]\times(0,\beta]$, we have $w(z_1)\neq w(z_2)$ because $X((0,a]\times(0,\beta])$ is a graph over $\Delta_4\setminus\{q_4\}$.  Let $\ell:[0,1]\rightarrow \Delta_4$ be the line segment connecting $p_1:=w(z_1)$ to $p_2:=w(z_2)$ with constant speed, that is, $\ell(0)=p_1$, $\ell(1)=p_2$ and $|\dot{\ell}(t)|=|p_2-p_1|$ for all $t\in[0,1]$.

(1) Choosing a fundamental region $\hat{\Sigma}_{q_4}$ of $\Sigma_{q_4}$ suitably, we may suppose $\ell$ is disjoint from $\pi(\partial\hat{\Sigma}_{q_4})$. Then there is a smooth curve $c:[0,1]\rightarrow(0,a]\times(0,2\beta]$ such that $\ell(t)=w(c(t))$. Clearly $|\dot{c}(t)|>0$ for all $0\leq t\leq1$. Let $g:[0,a]\times\mathbb R\rightarrow\mathbb C$ be the Gauss map of $\Sigma_{q_4}$. Krust showed that the inner product $W$ of the two vectors $p_2-p_1$ and $i(w^*(z_2)-w^*(z_1))$ of $\mathbb R^2$ is written as
$$W:=\langle p_2-p_1,i(w^*(z_2)-w^*(z_1))\rangle=\int_0^1\frac{1}{4}|\dot{c}(t)|^2\left(|g(c(t))|^2-\frac{1}{|g(c(t))|^2}\right)dt.$$
Since $\Sigma_{q_4}\setminus{{\gamma}}_{q_4}$ is a multi-graph, we have $|g|>1$ on $(0,a]\times\mathbb R$. Hence $W>0$ and therefore $w^*(z_1)\neq w^*(z_2)$.

(2) Suppose $\ell$ intersects $\pi(\partial\hat{\Sigma}_{q_4})$ at the point $q_4$. Then $c$ is piecewise smooth and there exist $0<d_1<d_2<1$ such that $q_4\notin w(c([0,d_1)))\cup w(c((d_2,1]))$, $w(c([d_1,d_2]))=\{q_4\}$, and $|\dot{c}(t)|>0$ for $t\in[0,d_1)\cup(d_2,1]$. Clearly $$|g(c(t))|=1\,\,{\rm for}\,\, t\in[d_1,d_2],\,\,\,\,|g(c(t))|>1\,\, {\rm for}\,\, t\in[0,d_1)\cup (d_2,1].$$ Hence
$$W=\left(\int_0^{d_1}+
\int_{d_2}^1\right)\,\frac{1}{4}|\dot{c}(t)|^2\left(|g(c(t))|^2-\frac{1}{|g(c(t))|^2}\right)dt>0$$
and so $w^*(z_1)\neq w^*(z_2)$.

Thus we can conclude that $X^*((0,a]\times(0,\beta))$ is a graph over the $x_1^*x_2^*$-plane. Since $X^*([0,a]\times\{0\})$ coincides with $X^*([0,a]\times\{\beta\})$, $X^*((0,a]\times[0,\beta])=\Sigma^*_{q_4}\setminus\gamma_{q_4}$ is also a graph over its projection into the $x_1^*x_2^*$-plane. This proves (e).
\end{proof}

\section{Pyramid}

It has been possible to construct free boundary minimal annuli in a tetrahedron $T$ because $T$ is the simplest polyhedron in $\mathbb R^3$. In general one cannot find a free boundary minimal annulus in a polyhedron like a quadrilateral pyramid $P_y$ in Figure 8. Of course, given a translationally periodic curve $\Gamma_q$ with fundamental piece $\bar{\gamma}_q\cup\bar{\gamma}_1$ corresponding to $c_1\nu_1,\ldots,c_5\nu_5$, respectively, where $\nu_1,\ldots,\nu_5$ are the unit normals to the faces of $P_y$, one can show that there exists a translationally periodic minimal surface $\Sigma_{q}$ spanning $\Gamma_q$. One can also find a point $q_5\in\Delta_5$ such that $\Sigma_{q_5}^*$ is a minimal annulus. However, $\Sigma_{q_5}^*$ may be a free boundary minimal annulus not in $P_y$ but in a polyhedron $P_o$ like Figure 8 which has the same unit normals as those of $P_y$. And yet, in case $P_y$ is a regular pyramid or a rhombic pyramid, we can show that $P_y$ has a free boundary minimal annulus. Surprisingly, we can also show that there exist genus zero free boundary minimal surfaces in every Platonic solid.

\begin{center}
\includegraphics[width=5.4in]{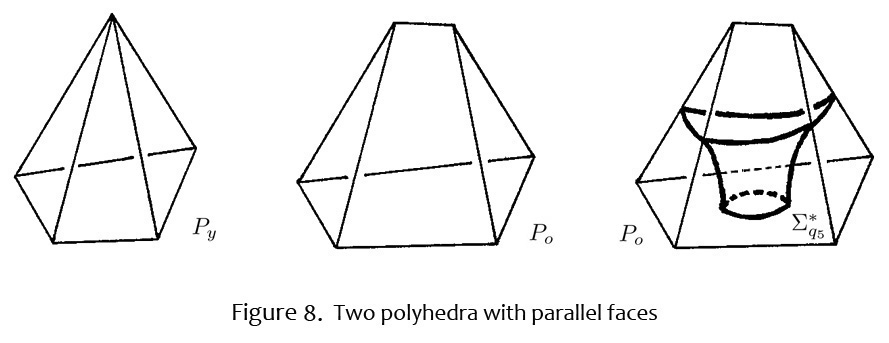}\\
\end{center}

\begin{theorem}\label{ps}
Let $P_y$ be a right pyramid whose base $B$ is a regular $n$-gon. Then there exists a free boundary minimal annulus $A$ in $P_y$ which is a graph over $B$. $A$ is invariant under the rotation by $2\pi/n$ about the line through the apex and the center of $B$. One component of $\partial A$ is convex and closed in $B$ and the other is convex in each remaining face of $P_y$.
\end{theorem}

\begin{proof}
Let  $F_1,\ldots,F_n$ be the faces of $P_y$ other than the base $B$. Denote by $\nu_0,\nu_1,\ldots,\nu_n$ the outward unit normals to $B,F_1,\ldots,F_n$, respectively. Then there exists a unique positive constant $c$ such that
$$c\nu_0+\nu_1+\cdots+\nu_n=0.$$
Assume that $B$ lies in the $x_1x_2$-plane with center at the origin. Let $\bar{\gamma}_0$ be a vertical line segment of length $c$\, on the $x_3$-axis and let $\bar{\gamma}_1$ be a connected piecewise linear curve determined by $\nu_1,\ldots,\nu_n$(i.e., $\nu_i$ is the $i$-th oriented line segment of $\bar{\gamma}_1$)  such that the projection $\pi(\bar{\gamma}_1)$ of $\bar{\gamma}_1$ onto the $x_1x_2$-plane is a regular $n$-gon centered at the origin. Moreover, let's assume that the two end points of $\bar{\gamma}_0$ and $\bar{\gamma}_1$ have the same $x_3$-coordinates: $0$ and $c$. $\bar{\gamma}_0\cup\bar{\gamma}_1$ determines a complete helically periodic curve $\Gamma$ of which $\bar{\gamma}_0\cup\bar{\gamma}_1$ is a fundamental piece. $\Gamma$ is translationally periodic as well. Then Theorem \ref{main} guarantees that there exists a translationally periodic minimal surface $\Sigma$ spanning $\Gamma$.

Define the screw motion $\sigma$ by
$$\sigma(r\cos\theta,r\sin\theta,x_3)=\left(r\cos(\theta+\frac{2\pi}{n}), r\sin(\theta+\frac{2\pi}{n}), x_3+\frac{c}{n}\right).$$
Obviously $\Sigma$ is invariant under $\sigma^n$. The point is that $\Sigma$ is invariant under $\sigma$ as well. This is because by Theorem \ref{plateau} the periodic Plateau solution spanning $\Gamma$ uniquely exists and $\sigma(\Sigma)$ also spans $\Gamma$. So evenly divide $\bar{\gamma}_0$ into $n$ line segments $\bar{\gamma}_{0}^1,\ldots,\bar{\gamma}_{0}^n$ such that
$$\bar{\gamma}_{0}^k:=\{p\in\bar{\gamma}_{0}:\frac{k-1}{n}c\leq x_3(p)\leq\frac{k}{n}c\},\,\,\,k=1\ldots, n.$$
Similarly, set
$$\Sigma^k=\{p\in\Sigma:\frac{k-1}{n}c\leq x_3(p)\leq\frac{k}{n}c\},\,\,\,k=1,\ldots,n.$$
It is clear that
$$\sigma(\bar{\gamma}_0^k)=\bar{\gamma}_0^{k+1},\,\,\sigma(\Sigma^k)=\Sigma^{k+1},\,\,k=1,\ldots,n-1,\,\,\,{\rm and}\,\,\, \sigma(\Sigma^n)=\sigma^n(\Sigma^1).$$
Denote by $f_\gamma(\Sigma)$ the flux of $\Sigma$ along $\gamma\subset\partial\Sigma$, that is,
$$f_\gamma(\Sigma)=\int_{p\in\gamma}n(p),$$
where $n(p)$ is the inward unit conormal to $\gamma$ on $\Sigma$ at $p\in\gamma$. Clearly
$$f_{\sigma(\gamma)}(\sigma(\Sigma))=\sigma(f_\gamma(\Sigma))\,\,\,{\rm and}\,\,\,f_{\bar{\gamma}_{0}}(\Sigma)=\sum_{k=1}^{n}f_{\bar{\gamma}_{0}^k}(\Sigma^k).$$
Hence
\begin{eqnarray*}
\sigma(f_{\bar{\gamma}_{0}}(\Sigma))&=&\sum_{k=1}^n\sigma(f_{\bar{\gamma}_{0}^k}(\Sigma^k))\,\,=\,\,\sum_{k=1}^n f_{\sigma(\bar{\gamma}_{0}^k)}(\sigma(\Sigma^k))\\
&=&\sum_{k=1}^{n-1}f_{\bar{\gamma}_{0}^{k+1}}(\Sigma^{k+1})+f_{\sigma^n(\bar{\gamma}^1_{0})}(\sigma^n(\Sigma^1))\\
&=&\sum_{k=1}^nf_{\bar{\gamma}_{0}^k}(\Sigma^k)\,\,=\,\,f_{\bar{\gamma}_{0}}(\Sigma).
\end{eqnarray*}
But $\sigma(f_{\bar{\gamma}_{0}}(\Sigma))=f_{\bar{\gamma}_{0}}(\Sigma)$  holds only when $f_{\bar{\gamma}_{0}}(\Sigma)=0$. In this case $f_{\bar{\gamma}_1}(\Sigma)$ also vanishes. Therefore $\Sigma^*$ is a well-defined minimal annulus.

We now show that $\Sigma^*$ is in $P_y$ with free boundary. Choose a point $p\in\Sigma^k$ with coordinates
$$X(p)=(x_1(p),x_2(p),x_3(p)).$$
Denote by $X^*(p)$ the point of $\Sigma^{k*}$ corresponding to $p\in\Sigma^k$,
$$X^*(p)=(x_1^*(p),x_2^*(p),x_3^*(p)).$$
The coordinates of $\sigma(p)$ are
$$X(\sigma(p))=\left((x_1(p),x_2(p))\cdot\left( \begin{array}{c}
\cos\alpha\\-\sin\alpha
\end{array}
\begin{array}{c}
\sin\alpha\\\cos\alpha
\end{array}\right),\,x_3(p)+\frac{c}{n}\right),\,\,\,\alpha=\frac{2\pi}{n}.$$
Then
\begin{eqnarray*}
X^*(\sigma(p))&=&\left((x_1^*(p),x_2^*(p))\cdot\left( \begin{array}{c}
\cos\alpha\\-\sin\alpha
\end{array}
\begin{array}{c}
\sin\alpha\\\cos\alpha
\end{array}\right),\,x_3^*(p)+0\right)\\
&=&\sigma_0(X^*(p)),
\end{eqnarray*}
where $\sigma_0$ is the rotation in $\mathbb R^3$ defined by
$$\sigma_0(r\cos\theta,r\sin\theta,x_3)=\left(r\cos(\theta+\frac{2\pi}{n}), r\sin(\theta+\frac{2\pi}{n}), x_3\right).$$
Hence
\begin{equation}\label{rot}
(\Sigma^{k+1})^*=\sigma_0(\Sigma^{k*}),\,\,k=1,\ldots,n
\end{equation}
and so
\begin{eqnarray*}
\sigma_0(\Sigma^*)&=&\sigma_0(\Sigma^{1*}\cup\cdots\cup\Sigma^{n*})\,\,=\,\,\Sigma^{2*}\cup\cdots\cup\Sigma^{n*}\cup\sigma_0(\Sigma^{n*})\\
&=&\Sigma^{2*}\cup\cdots\cup\Sigma^{n*}\cup\sigma_0^n(\Sigma^{1*})\,\,=\,\,\Sigma^*.
\end{eqnarray*}
Therefore $\Sigma^*$ is invariant under the rotation $\sigma_0$.
We know that the curve $X^*(\nu_1)$ is in the plane $\{y_1^*=d_1^*\}$ orthogonal to $\nabla y_1^*=\nu_1$ and $\Sigma^*$ is perpendicular to that  plane along $X^*(\nu_1)$.
Therefore \eqref{rot} implies that $\Sigma^*$ is a free boundary minimal surface in the pyramid $P_m$ bounded by a plane perpendicular to $\nu_{n+1}$ and by the $n$ planes $\cup_{i=1}^{n}(\sigma_0)^i(\{y_1^*=d_1^*\})$. $P_m$ is similar to $P_y$ and a homothetic expansion $A$ of $\Sigma^*$ is a free boundary minimal annulus in $P_y$.
By the same argument as in the proof of Theorem \ref{fb} we see that $A$ is a graph over $B$ and $\partial A$ is convex on each face of $P_y$.

There is another way of constructing $\Sigma^*$: Smyth's method. Divide the regular $n$-gon $B$ into $n$ congruent triangles $B_1,\ldots,B_n$. Then one can tessellate $P_y$ by $n$ congruent tetrahedra $T_1,\ldots,T_n$ with the apex of $P_y$ as their common vertex and $B_1,\ldots,B_n$ as their bases. Smyth's theorem gives us three free boundary minimal disks in $T_1$. Among them let's choose the one that is disjoint from the line through the apex and the center of $B$. By reflections one can extend the chosen minimal disk to a free boundary minimal annulus in $P_y$. This annulus must be the same as $A$ by the uniqueness of Theorem \ref{plateau} (c).
\end{proof}

\begin{corollary}\label{pl}
Every Platonic solid with regular $n$-gon faces  has five types of embedded, genus zero, free boundary minimal surfaces $\Sigma_1,\ldots,\Sigma_5$. Three of them, $\Sigma_1,\Sigma_2,\Sigma_3$, intersect each face along 1, $n$, $2n$ closed convex congruent curves, respectively. $\Sigma_4$ intersects every edge of the solid and $\Sigma_5$ surrounds every vertex of the solid. {\rm (See Figure 3.)}
\end{corollary}

\begin{proof}
Given a Platonic solid $P_s$, let $p$ be its center and $F$ one of its faces. Then the cone from $p$ over $F$ is a right pyramid with a regular $n$-gon base and hence $P_s$ is tessellated into congruent pyramids. Each pyramid contains an embedded free boundary minimal annulus by Theorem \ref{ps}. The union of all those minimal annuli in the congruent pyramids of the tessellation, denoted as $\Sigma_1$, is the analytic continuation of each minimal annulus into an embedded, genus zero, free boundary minimal surface in $P_s$.

The regular $n$-gon $F$ can be tessellated into $n$ isosceles, one of which is denoted as $F_n$. $F_n$ can be divided into two congruent right triangles, one of which is $F_{2n}$. Then the cone from $p$ over $F_n$ is a tetrahedron $T_n$ and  $T_{2n}$ denotes the tetrahedron determined by $p$ and $F_{2n}$. Note here that all the dihedral angles of $T_{2n}$ are $\leq90^\circ$ whereas an edge of $T_n$ has dihedral angle $=120^\circ$ in case $P_s$ is a tetrahedron, an octahedron or an icosahedron. Fortunately, the three dihedral angles of $T_n$ along $\partial F_n$ are $\leq90^\circ$. Hence  by Theorem \ref{fb} there exist free boundary minimal annuli $A_2$ in $T_n$ and $A_3$ in $T_{2n}$ one boundary component of which is a closed convex curve in $F_n$ and in $F_{2n}$, respectively. Then $\Sigma_2,\Sigma_3$ are exactly the analytic continuations (by reflection) of $A_2,A_3$, respectively.

On the other hand, Smyth's theorem gives three minimal disks $S_4,S_{5},S_6$ with free boundary in $T_{2n}$. Only one of them, say $S_6$, is disjoint from $\partial F$. Then $S_6$ must be a subset of $\Sigma_1$. Since $P_s$ is tessellated by congruent copies of $T_{2n}$, both $S_4$ and $S_5$ can be extended analytically into free boundary embedded minimal surfaces of genus zero in $P_s$, which we denote as $\Sigma_4$ and $\Sigma_5$. Assuming that $S_4$ connects the two orthogonal edges of $F_{2n}$, we see that every boundary component of $\Sigma_4$ intersects exactly one edge of $P_s$ orthogonally. Then $S_5$ connects two nonorthogonal edges of $F_{2n}$ and hence each boundary component of $\Sigma_5$ surrounds exactly one vertex of $P_s$.
\end{proof}

\begin{corollary}\label{pr}
If $P_r$ is a right pyramid with rhombic base $B$, there exists a free boundary minimal annulus $A$ in $P_r$ which is a graph over $B$. One boundary component of $A$ is  convex and closed in $B$ and the other one is convex in each remaining face of $P_r$.
\end{corollary}
\begin{proof}
Similar to Theorem \ref{ps}
\end{proof}

\begin{remark}
(a) As $n\rightarrow\infty$, $P_y$ of Theorem \ref{ps} becomes a right circular cone and then $\Sigma$ will be part of the helicoid and $\Sigma^*$ a catenoidal waist in $P_y$.

(b) In case the Platonic solid $P_s$ is a cube, the free boundary minimal surface $\Sigma_1$ with genus 0 proved to exist in $P_s$ by Corollary \ref{pl} is the same as Schwarz's $P$-surface $S$. This can be verified as follows. The cube $P_s$ is tessellated into six right pyramids with square base. Let $P_y$ be one of them. Then $(\Sigma_1\cap P_y)^*$ and $(S\cap P_y)^*$ are translationally periodic minimal surfaces. Denote their boundaries by $\Gamma_{\Sigma_1}$ and $\Gamma_S$, respectively. $\Gamma_{\Sigma_1}$ and $\Gamma_S$ are piecewise linear and translationally periodic. Since their fundamental pieces are determined by the outward unit normals to the faces of the same pyramid $P_y$, $\Gamma_{\Sigma_1}$ and $\Gamma_S$ must be identical. Let $\bar{\gamma}_0\cup\bar{\gamma}_1$ be their fundamental piece. Since the projection of $\bar{\gamma}_1$ into the base of $P_y$ is a square which is convex, there is only one periodic Plateau solution spanning $\Gamma_{\Sigma_1}$ by Theorem \ref{plateau} (c). Hence $\Sigma_1^*$ must be the same as $S^*$ and therefore  $\Sigma_1=S$. Similarly, the free boundary minimal surface $\Sigma_1$ in the regular tetrahedron is the same as the one constructed by Nitsche \cite{N2}.

(c) In the cube $P_s$, $\Sigma_4$ is nothing but {\it Neovius' surface} and $\Sigma_5$ is {\it Schoen's I-WP surface} (see figure 9). Only in the cube can one construct an extra free boundary minimal surface $\Sigma_6$ as follows. Let $F$ be a square face of the cube and let $F_2$ be a right isosceles which is a half of $F$. Then the tetrahedron that is the cone from the center of $P_s$ over $F_2$ contains three free boundary minimal disks. If we choose one of the three that connects the two orthogonal edges of $F_2$, then its analytic continuation is the desired $\Sigma_6$. This is {\it Schoen's F-RD surface} which surrounds only four vertices of the cube whereas Schoen's I-WP surface surrounds all eight vertices of the cube (see Figure 9).

\begin{center}
\includegraphics[width=4.7in]{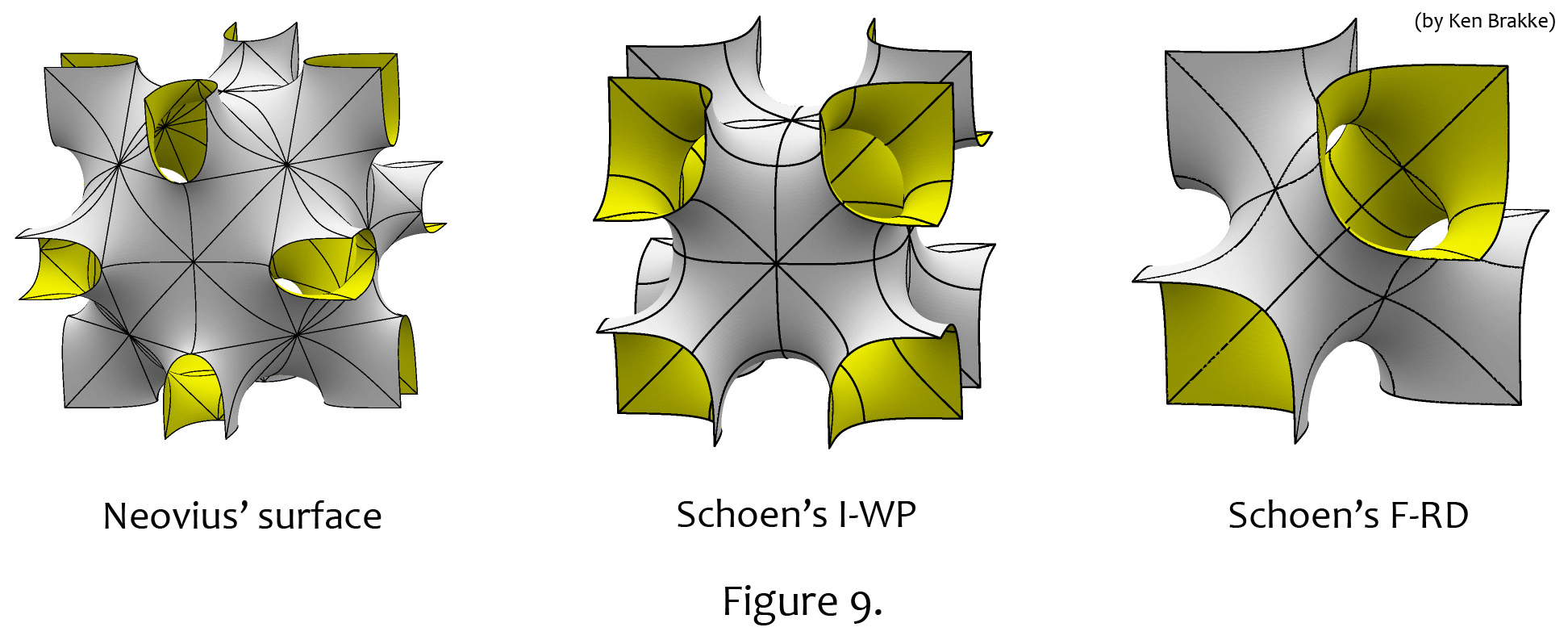}\\
\end{center}
\end{remark}

We would like to conclude our paper by proposing the following interesting problems.\\

\noindent{\bf Problems.}
\begin{enumerate}
\item  What kind of a pyramid $P_y$ with $n$-gon  base ($n\geq4$) has a free boundary minimal annulus?

\item  Let $\Gamma$ be a Jordan curve in $\mathbb R^3$ bounding a minimal disk $\Sigma$. If the total curvature of $\Gamma$ is $\leq4\pi$, we know  that $\Sigma$ is unique \cite{N}. Show that $\Sigma^*$ is the unique minimal disk spanning $\partial\Sigma^*$.

\item  Assume that $\Gamma\subset\mathbb R^3$ is a Jordan curve with total curvature $\leq4\pi$. It is proved that any minimal surface $\Sigma$ spanning $\Gamma$ is embedded \cite{EWW}. If $\Sigma$ is simply connected, show that $\Sigma^*$ is also embedded.

\item Let $\Gamma$ be a complete translationally (or helically) periodic curve with a fundamental piece $\bar{\gamma}$. Assume that a translationally(or helically) periodic minimal surface $\Sigma_\Gamma$ spans $\Gamma$. What is the maximum total curvature of $\bar{\gamma}$ that guarantees the uniqueness of $\Sigma_\Gamma$? What about the embeddedness of $\Sigma_\Gamma$?

\item  Assume that $\Sigma\subset\mathbb R^3$ is a free boundary minimal annulus in a ball. Show that $\Sigma^*$ is a translationally periodic free boundary minimal surface in a cylinder so that $\Sigma$ is necessarily the critical catenoid.
\end{enumerate}

\end{document}